\documentclass[14pt,reqno]{amsart}
\usepackage[dvipdfmx]{color,graphicx,hyperref}
\usepackage{amsfonts,amsthm,amsmath,amssymb,bm,mathrsfs}
\newtheorem{theorem}{Theorem}
\newtheorem{corollary}[theorem]{Corollary}
\newtheorem{prop}[theorem]{Proposition}
\newtheorem{remark}[theorem]{Remark}
\newtheorem{lemma}[theorem]{Lemma}
\def\({\left(}
\def\){\right)}
\def\re{\mathbb{R}}
\def\N{\mathbb{N}}
\def\Sp{\mathbb{S}}
\def\pd{\partial}
\def\disp{\displaystyle}
\def\intRN{\int_{\mathbb{R}^N}}
\def\tcr{\relax}
\def\bu{{\bm u}}
\def\bv{{\bm v}}
\def\bx{{\bm x}}
\def\tl{\tageq\label}
\newcommand*\tageq{\refstepcounter{equation}\tag{\theequation}}
\begin{document}
\title[Rellich-Hardy inequality for curl-free fields]{A curl-free improvement of the Rellich-Hardy inequality with weight\\}
\author[N. Hamamoto \and F. Takahashi]{Naoki Hamamoto \and Futoshi Takahashi}
\address{%
\begin{tabular}{l}
Osaka City University Advanced Mathematical Institute, 
\\ 3-3-138 Sugimoto, Sumiyoshi-ku, Osaka 558-8585, Japan
\end{tabular}
}
\email{yhjyoe@yahoo.co.jp {\rm (N.Hamamoto)}} 
\email{futoshi@sci.osaka-cu.ac.jp {\rm (F.Takahashi)}}

\begin{abstract}
We consider the best constant in the Rellich-Hardy inequality (with a radial power weight) for curl-free vector fields on $\mathbb{R}^N$, 
originally found by Tertikas-Zographopoulos \cite{Tertikas-Z} for unconstrained fields. 
\tcr{This inequality is considered as} an intermediate between Hardy-Leray and Rellich-Leray inequalities. 
Under the curl-free condition, we compute the new explicit \tcr{best constant in the inequality and prove the non-attainability of the constant.} 
This paper is a sequel to \cite{CF_MAAN,CF_Re}.  
\end{abstract}

\subjclass[2020]{Primary 26D10; Secondary 35A23.}

\keywords{Rellich-Hardy inequality, curl-free vector fields, best constant.}


\maketitle

\section{Introduction}

Let $N \in \N$ be an integer with $N \ge 2$, let $\gamma \in \re$ and put $\bx=(x_1,\cdots,x_N) \in \re^N$. 
In the following, the notation $\mathcal{D}_{\gamma}(\mathbb{R}^N)$ denotes the set of \tcr{compact supported,} smooth vector fields on $\mathbb{R}^N$ 
\[
\bu = (u_1,u_2,\cdots, u_N): \mathbb{R}^N \ni \bx \mapsto \bu(\bx) \in \re^N
\]
with the integrability condition \[\int_{\mathbb{R}^N}|{\bm u}|^2|{\bm x}|^{2\gamma-2}dx<\infty.\] 
\subsection{Preceding results and motivation} 

It is well known that the {\it Hardy-Leray inequality}
\begin{equation}
 \(\gamma + \tfrac{N}{2}-1 \)^2 \intRN \frac{|\bu|^2}{|\bx|^2} |\bx|^{2\gamma} d{x} \le \intRN |\nabla \bu|^2 |\bx|^{2\gamma} d{x}
\label{Hardy}
\end{equation}
holds for any vector field $\bu \in \mathcal{D}_\gamma(\mathbb{R}^N)$, with the best constant on the left-hand side.
This was first proved by J. Leray \cite{Leray} for $(N,\gamma)=(3,0)$, as a higher-dimensional extension of Hardy's inequality in one dimension \tcr{\cite{Hardy}}, see also the book by Ladyzhenskaya \cite{Ladyzhenskaya}. 
Costin and Maz'ya \cite{Costin-Mazya} improved the best value of the constant by assuming ${\bm u}$ to be divergence-free (under the additional assumption of axisymmetry for $N\ge3$):  for the case $N=2$, it was shown that the inequality
\[C_\gamma \int_{\mathbb{R}^2}\frac{|{\bm u}|^2}{|{\bm x}|^2}|{\bm x}|^{2\gamma}dx\le\int_{\mathbb{R}^2}|\nabla {\bm u}|^2|{\bm x}|^{2\gamma}dx\]
holds with the best constant $C_\gamma=\begin{cases} \frac{3+(\gamma-1)^2}{1+(\gamma-1)^2}\gamma^2&\(|\gamma+1|\le\sqrt{3}\)\\ \gamma^2+1&\(|\gamma+1|>\sqrt{3}\)\end{cases}$ for divergence-free vector fields ${\bm u}\in\mathcal{D}_\gamma(\mathbb{R}^2)$. 
Since there is an isometry on $\mathbb{R}^2$ between divergence-free fields and curl-free fields, 
the same conclusion also holds for the two-dimensional curl-free fields. 
As a generalization of this result, we have derived in recent papers \cite{CF_MAAN,CF_Re} the Hardy-Leray inequality
\begin{equation}
\label{HL}
 H_{N,\gamma}\int_{\mathbb{R}^N}\frac{|{\bm u}|^2}{|{\bm x}|^2}|{\bm x}|^{2\gamma}dx\le\int_{\mathbb{R}^N}|\nabla {\bm u}|^2|{\bm x}|^{2\gamma}dx
\end{equation}
for curl-free fields ${\bm u}\in\mathcal{D}_\gamma(\mathbb{R}^N)$ with the best constant 
\begin{equation}
\label{HNg}
H_{N,\gamma} = 	\begin{cases}	\(\gamma + \frac{N}{2}-1\)^2 \frac{3(N-1) + \(\gamma + \frac{N}{2}-2 \)^2}{N-1 + \(\gamma+\frac{N}{2}-2\)^2} & \text{if} \quad \big| \gamma + \frac{N}{2} \big| \le \sqrt{N+1},\\ 	\(\gamma + \frac{N}{2}-1\)^2 + N-1 & \text{otherwise.}
\end{cases}\end{equation}
Since $C_\gamma=H_{2,\gamma}$, this result recovers Costin-Maz'ya's one for $N = 2$. 

The {\it Rellich-Leray inequality} is given by
\begin{equation} 
\label{Rellich}
	B_{N,\gamma}\int_{\mathbb{R}^N}\frac{|{\bm u}|^2}{|{\bm x}|^4}|{\bm x}|^{2\gamma}dx\le\int_{\mathbb{R}^N}|\triangle {\bm u}|^2|{\bm x}|^{2\gamma}dx
\end{equation}
for unconstrained fields ${\bm u}\in \mathcal{D}_{\gamma-1}(\mathbb{R}^N)$, where the constant $B_{N,\gamma}$  is sharp when 
\begin{equation} 
\label{RLB}
B_{N,\gamma}=\min_{\nu\in\mathbb{N}\cup\{0\}}\(\(\gamma-1\)^2-\(\nu+\tfrac{N}{2}-1\)^2\)^2.
\end{equation}
This was found by Rellich \cite{Rellich} for $\gamma=0$ and  Caldiroli-Musina \cite{Caldiroli} for $\gamma\ne 0$. 
In recent papers \cite{CF_MAAN,CF_Re}, we additionally considered the curl-free improvement of the Rellich-Leray inequality:  if ${\bm u}\in\mathcal{D}_{\gamma-1}(\mathbb{R}^N)$ is assumed to be curl-free, then the same inequality \eqref{Rellich} holds with the best constant $B_{N,\gamma}$ \tcr{:} 
\begin{equation}
\label{RL}
 B_{N,\gamma}=\min\left\{\((\gamma-1)^2-\tfrac{N^2}{4}\)^2,\ \min_{\nu\in\mathbb{N}} \tfrac{\(\gamma+\frac{N}{2}-1\)^2+\alpha_\nu}{\(\gamma+\frac{N}{2}-3\)^2+\alpha_\nu}\((\gamma-2)^2-\(\nu+\tfrac{N}{2}-1\)^2\)^2\right\};
\end{equation}
here and hereafter we use the notation 
\begin{equation}
\label{alpha_s}
	 \alpha_{s}={s}({s}+N-2)
\end{equation}
for any ${s}\in\mathbb{R}$. 

In this paper, \tcr{we are interested in} 
another version of Rellich-Leray inequality \cite{Beckner, GM, Tertikas-Z}: 
\begin{equation}
\label{RH}
 \int_{\mathbb{R}^N}|\triangle {\bm u}|^2|{\bm x}|^{2\gamma}dx\ge A_{N,\gamma}\int_{\mathbb{R}^N}\frac{|\nabla {\bm u}|^2}{|{\bm x}|^2}|{\bm x}|^{2\gamma}dx,\qquad\forall {\bm u}\in\mathcal{D}_{\gamma-1}(\mathbb{R}^N)
\end{equation}
holds with the best constant $\displaystyle A_{N,\gamma}=\min_{\nu\in\mathbb{N}\cup\{0\}}A_{N,\gamma,\nu}$, where 
\begin{equation}
\label{ANg}
A_{N,\gamma,\nu}:=\left\{\begin{array}{cl} \(\gamma-\frac{N}{2}\)^2 &\text{for }\nu=0,\vspace{0.5em}\\\frac{\(\(\gamma-1\)^2-\(\nu+\frac{N}{2}-1\)^2 
\)^2}{\(\gamma+\frac{N}{2}-2\)^2+\alpha_\nu}\ 
&\text{for }\nu\in\mathbb{N}.\end{array}\right.
\end{equation}
We call \eqref{RH} the  {\it Rellich-Hardy inequality}. 
This inequality was first found for $N\ge5$ by Tertikas-Zographopoulos \cite[Theorem 1.7]{Tertikas-Z}. 
Subsequently, Beckner \cite{Beckner} and Ghoussoub-Moradifam \cite{GM} established the same inequality when $N \in \{3, 4\}$ and $\gamma = 0$, 
with the best constants $A_{3,0} = \frac{25}{36}$ and $A_{4,0} = 3$.
See also Cazacu \cite{Cazacu_2019} for the unified proof of the inequality when $\gamma = 0$. 

Now, let us assume that $\gamma$ satisfies 
\[
	A_{N,\gamma}=\min_{\nu\in\mathbb{N}\cup\{0\}}A_{N,\gamma,\nu}=A_{N,\gamma,0} = \(\gamma-\tfrac{N}{2}\)^2.
\]
Then we see that a successive application of Rellich-Hardy and Hardy-Leray inequalities  reproduces Rellich-Leray inequality \eqref{Rellich}: we have
\[
\begin{split}
  \mathop{\int}_{\mathbb{R}^N}|\triangle {\bm u}|^2|{\bm x}|^{2\gamma}dx
&\ge A_{N,\gamma,0}\mathop{\int}_{\mathbb{R}^N}|\nabla {\bm u}|^2|{\bm x}|^{2(\gamma-1)}dx\qquad\qquad \ \ \qquad\(\text{from \eqref{RH}}\)
 \\&\ge A_{N,\gamma,0}\(\gamma+\tfrac{N}{2}-2\)^2\mathop{\int}_{\mathbb{R}^N}\frac{|{\bm u}|^2}{|{\bm x}|^2}|{\bm x}|^{2(\gamma-1)}dx
\quad \(\begin{array}{ll}
\small	 \text{from \eqref{Hardy}}
\\
\small\text{with $\gamma$ replaced by $\gamma-1$}\end{array}
\)
 \\&=\(\(\gamma-1\)^2-\(\tfrac{N}{2}-1\)^2\)^2\int_{\mathbb{R}^N}\frac{|{\bm u}|^2}{|{\bm x}|^4} |{\bm x}|^{2\gamma} dx
\\&\ge B_{N,\gamma}\int_{\mathbb{R}^N}\frac{|{\bm u}|^2}{|{\bm x}|^4} |{\bm x}|^{2\gamma} dx
\end{split}
\]
with $B_{N,\gamma}$ given by \eqref{RLB}. 
Hence,  Rellich-Hardy inequality can be considered as a stronger version of Rellich-Leray inequality 
and plays a role as an intermediate between Rellich-Leray and Hardy-Leray inequalities. 

In the context of the curl-free improvement, it seems also natural to ask whether 
the same phenomenon \tcr{will happen in \eqref{RH}}; in particular, \tcr{the main interest of our problem} is how the best constant in \eqref{RH} will change when ${\bm u}$ is assumed to be curl-free.

\subsection{Results}
Motivated by the observation above, we aim to derive the best constant in Rellich-Hardy inequality for curl-free fields. 
Now, our main result reads as follows:
\begin{theorem}
\label{theorem:RH}
Let $N \ge 2$. 
Let ${\bm u}\in\mathcal{D}_{\gamma-1}(\mathbb{R}^N)$ be a curl-free vector field. Then the inequality
\begin{equation}
\label{RH_cf}
\int_{\re^N}|\triangle {\bm u}|^2|{\bm x}|^{2\gamma}dx
\ge C_{N,\gamma}\int_{\re^N}\frac{|\nabla{\bm u}|^2}{|{\bm x}|^2}|{\bm x}|^{2\gamma}dx
\end{equation}
holds with the best constant $C_{N,\gamma}$ 
expressed as
\[
 C_{N,\gamma}=\min_{\nu\in\mathbb{N}\cup\{0\}}C_{N,\gamma,\nu},
\]
where
\begin{align*}
\tl{C0A1}  &   C_{N,\gamma,0}=\frac{\((\gamma-1)^2-\frac{N^2}{4}\)^2}{\(\gamma+\frac{N}{2}-2\)^2+N-1}=A_{N,\gamma,1}
  ,
  \\
   &   C_{N,\gamma,1}=
  \frac{\(\gamma-\frac{N}{2}-2\)^2\(\(\gamma+\frac{N}{2}-1\)^2+N-1\)}{\(\gamma+\frac{N}{2}-3\)^2+3(N-1)}
 \end{align*}
and
\[
 C_{N,\gamma,\nu}=
\tfrac{\((\gamma-2)^2-\(\nu+\frac{N}{2}-1\)^2\)^2\(\(\gamma+\frac{N}{2}-1\)^2+\alpha_\nu\)}{\((\gamma-2)^2-\(\nu+\frac{N}{2}-1\)^2\)^2+2(\gamma-1)\(\(2\gamma+N-5\)\alpha_\nu+(N-1)\(\gamma+\frac{N}{2}-3\)^2\)}
\quad\text{ for }\nu\ge2.
\]
\end{theorem}
Moreover, we obtain a stronger inequality by adding a remainder term to the right-hand side of \eqref{RH_cf}.
\begin{theorem}
\label{theorem:RH_c}
Let $C_{N,\gamma}$ be the same constant as in Theorem \ref{theorem:RH}. Then there exists an absolute constant $c>0$ such that the inequality
\begin{equation}
\label{RH_cfr}
\begin{split}
 \int_{\mathbb{R}^N}|\triangle {\bm u}|^2|{\bm x}|^{2\gamma}dx&-C_{N,\gamma}\int_{\mathbb{R}^N}\frac{|\nabla {\bm u}|^2}{|{\bm x}|^2}|{\bm x}|^{2\gamma}dx\\&\ge {c}\int_{\mathbb{R}^N}\left|\nabla\(|{\bm x}|^{2-\frac{N}{2}-\gamma}({\bm x}\cdot\nabla) \big(|{\bm x}|^{\gamma+\frac{N}{2}-2}{\bm u}\big)\)\right|^2|{\bm x}|^{2\gamma-2}dx
\end{split}
\end{equation}
holds for all curl-free fields ${\bm u}\in\mathcal{D}_{\gamma-1}(\mathbb{R}^N)$.
\end{theorem}

\begin{remark}
From the proof 
below, we see that the constant $c$ on the right-hand side of \eqref{RH_cfr} can be  estimated by 
 \begin{alignat*}{3}
  &c\ge1  &&\text{\rm when}\quad  \gamma\le1, 
  \\
  &c \ge 1/2 &&\text{\rm when}\quad  N \ge 3 \quad\text{\rm and}\quad \gamma>1,
  \\
  &c \ge 1/3 \quad&&\text{\rm when}\quad N=2\quad\text{\rm and}\quad  \gamma>1 .
 \end{alignat*}
However, the best possible (the largest) value of $c$ is unknown.
\end{remark}

As a direct consequence of Theorem \ref{theorem:RH_c}, we can conclude that the best constant $C_{N,\gamma}$ of the inequality \eqref{RH_cf} is never attained in $\mathcal{D}_{\gamma-1}(\mathbb{R}^N)\setminus\{{\bm 0}\}$: 

\begin{corollary}
\label{corollary:RH_c}
If the equation
\[
	\int_{\re^N}|\triangle {\bm u}|^2|{\bm x}|^{2\gamma}dx = C_{N,\gamma}\int_{\re^N}\frac{|\nabla{\bm u}|^2}{|{\bm x}|^2}|{\bm x}|^{2\gamma}dx
\]
holds for a curl-free field ${\bm u}\in\mathcal{D}_{\gamma-1}(\mathbb{R}^N)$, then ${\bm u}\equiv {\bm 0}$.
\end{corollary}

\begin{proof}
Let \eqref{RH_cf} holds true. Then the right-hand side of \eqref{RH_cfr} must vanish. 
Thus 
\[
	{r}\partial_{r}\({r}^{\gamma+\frac{N}{2}-2}{\bm u}({r}{\bm \sigma})\) ={r}^{\gamma+\frac{N}{2}-2}{\bm x}_0
\]
holds for some constant vector filed ${\bm x_0}$, where $({r},{\bm \sigma})=(|{\bm x}|,{\bm x}/|{\bm x}|)$.
 Integrating both sides on any interval $[{s},{r}]\subset\mathbb{R}_+$ with respect to the measure $\frac{1}{{r}}d{r}$,
we have
\[
  {\bm u}({r}{\bm \sigma})=
  \begin{cases}
   \(\frac{{s}}{{r}}\)^{\gamma+\frac{N}{2}-2}{\bm u}({s}{\bm \sigma})+\frac{\(\frac{{s}}{{r}}\)^{\gamma+\frac{N}{2}-2}-1}{\gamma+\frac{N}{2}-2}{\bm x}_0&(\gamma\ne2-\tfrac{N}{2})
\vspace{0.5em}   \\
   {\bm u}({s}{\bm \sigma})+{\bm x}_0\log \frac{{s}}{{r}}&(\gamma=2-\tfrac{N}{2})
  \end{cases}.
\]
 In the case $\gamma\ne2-\frac{N}{2}$, take the limit $s \to \infty$ (resp. $s \to +0$) when $\gamma<2-\frac{N}{2}$ (resp. $\gamma>2-\frac{N}{2}$),
then we obtain
\[
 {\bm u}({r}{\bm \sigma})\equiv \tfrac{-1}{\gamma+\frac{N}{2}-2}{\bm x}_0\quad\text{ and hence }\quad {\bm u}({\bm x})\equiv {\bm u}({\bm 0}).
\]
 This fact together with the integrability condition $\int_{\mathbb{R}^N}|{\bm u}|^2|{\bm x}|^{2\gamma-4}dx<\infty$ says that
 ${\bm u}({\bm 0})$ must vanish, whence ${\bm u}\equiv{\bm 0}$. 
In the case $\gamma= 2-\frac{N}{2}$,  taking  ${s}\to0$ leads to
\[
 {\bm u}({r}{\bm \sigma})={\bm u}({\bm 0})+{\bm x}_0\lim_{{s}\to0}\log \frac{{s}}{{r}}
\]
and the finiteness of the right-hand side yields ${\bm x}_0={\bm 0}$. 
 Therefore,  we see again that ${\bm u}\equiv {\bm u}({\bm 0})$ and hence ${\bm u}\equiv {\bm 0}$. 
\end{proof}

As another direct consequence  of Theorem \ref{theorem:RH_c}, we have the following fact: 
\begin{corollary}
Let $C_{N,\gamma}$ be the same constant as in Theorem \ref{theorem:RH}. Then there exists an absolute constant $c>0$ such that the inequality
\begin{equation*}
\begin{split}
 \int_{\mathbb{R}^N}|\nabla \triangle \phi|^2|{\bm x}|^{2\gamma}dx&-C_{N,\gamma}\int_{\mathbb{R}^N}\frac{|D^2 \phi|^2}{|{\bm x}|^2}|{\bm x}|^{2\gamma}dx
\\&\ge {c}\int_{\mathbb{R}^N}\left|\nabla\(|{\bm x}|^{2-\frac{N}{2}-\gamma}({\bm x}\cdot\nabla) \big(|{\bm x}|^{\gamma+\frac{N}{2}-2} \nabla \phi \big)\)\right|^2|{\bm x}|^{2\gamma-2}dx
\end{split}
\end{equation*}
holds for all scalar field $\phi$ such that $\nabla \phi \in\mathcal{D}_{\gamma-1}(\mathbb{R}^N)$.
Here $D^2 \phi$ denotes a Hessian matrix of $\phi$.
\end{corollary}

\subsection*{Overview of the remaining content of the present paper}


The rest of this paper is organized as follows:
Section \ref{section2} provides a minimum required notations and definitions, and reviews a representation of curl-free fields. 
Section \ref{section3} gives the proof of Theorem \ref{theorem:RH}: we recall from \cite{CF_Re} the scalar-potential expression of $L^2$ integrals of curl-free fields; 
after that, we derive Lemma \ref{lemma:Q/P} as a key tool for evaluating the ratio of the two integrals in \eqref{RH_cf}, which also plays a computational part in the proof of Theorem \ref{theorem:RH_c}.
The proof of Lemma \ref{lemma:Q/P} is separated into two cases. Since both the cases use similar techniques and consist of long calculations, we prove only one case in the same section, and  postpone the other case in Section \ref{sec:Q/P}.
Section \ref{sec:RH_c} proves Theorem \ref{theorem:RH_c} by using an operator-polynomial representation of Rellich-Hardy integral quotient and by making full use of Lemma \ref{lemma:Q/P}. 
Section \ref{sec:observation} observes curl-free improvement phenomena of best constants in some cases.

\section{Preliminary for the proof of main theorem}
\label{section2}

\subsection{Notations and definitions in vector calculus on $\dot{\mathbb{R}}^N$}
Here we summarize the minimum required notations and definitions for the proof of our main theorems. 
We basically use the notation
\[
 \dot{\mathbb{R}}^N=\mathbb{R}^N\setminus\{{\bm 0}\}\qquad\text{and}\qquad\mathbb{S}^{N-1}=\left\{{\bm x}\in\mathbb{R}^N :\ |{\bm x}|=1\right\}.
\]
For every vector ${\bm x}\in\dot{\mathbb{R}}^N$, the notation
\[
 {r}=|{\bm x}|>0,\quad {\bm \sigma}={\bm x}/|{\bm x}|\in\mathbb{S}^{N-1}
\]
denotes the radius of ${\bm x}$  and its unit-vector part, which defines the smooth transformation 
\[
\dot{\mathbb{R}}^N\to\re_+\times\Sp^{N-1} , \quad 
	{\bm x}\mapsto ({r},{\bm \sigma})
\]
together with its inverse
\[
\mathbb{R}_+\times\mathbb{S}^{N-1}
\to
\dot{\mathbb{R}}^N ,\quad 
 ({r},{\bm \sigma})\mapsto{r}{\bm \sigma}.
  \]
Every vector field ${\bm u} = (u_1,u_2,\cdots,u_N):\dot{\mathbb{R}}^N \to \re^N$ has its radial scalar component $u_R=u_R({\bm x})$ and spherical vector part ${\bm u}_S={\bm u}_S({\bm x})$ given by the formulae
\[
	{\bm u}={\bm \sigma}u_R + {\bm u}_S ,\quad\ {\bm \sigma} \cdot {\bm u}_S = 0
\]
for all ${\bm x}\in\dot{\mathbb{R}}^N$;  the two fields are explicitly given by \[u_R={\bm \sigma}\cdot {\bm u} \quad \text{ and }\quad {\bm u}_S={\bm u}-{\bm \sigma}u_R. \]
In a similar way, the gradient operator $\nabla=\(\frac{\partial}{\partial x_1},\cdots,\frac{\partial}{\partial x_N}\)$ can be decomposed into the radial derivative $\partial_{r}$ and the spherical gradient $\nabla_{\!\sigma}$ as
\[
 \nabla={\bm \sigma}\partial_{r}+\frac{1}{{r}}\nabla_{\!\sigma},
\]
in order that $\partial_{r}f=(\nabla f)_R={\bm \sigma}\cdot\nabla f$ and $\frac{1}{{r}}\nabla_{\!\sigma}f=(\nabla f)_S$ for all $f\in C^\infty(\dot{\mathbb{R}}^N)$.  
The notation
\begin{equation}
\label{der} \partial: ={r}\partial_{r}={\bm x}\cdot\nabla
\end{equation}
denotes an alternative radial derivative, in order that the above decomposition formula of $\nabla$ can be rewritten as
\[{r}\nabla={\bm \sigma}\partial+\nabla_{\!\sigma}.\] 
The Laplace operator $\triangle = \sum_{k=1}^N \pd^2/\pd x_k^2$ is known to be expressed in terms of $({r},{\bm \sigma})$ by the formula 
\[
	 \triangle = \frac{1}{{r}^{N-1}} \pd_{r} \( {r}^{N-1} \pd_{r} \) + \frac{1}{{r}^2} \triangle_\sigma 
=\frac{1}{r^2}\(\partial^2 + (N-2)\partial + \triangle_\sigma\),
\]
where $\triangle_\sigma$ denotes the Laplace-Beltrami operator on $\Sp^{N-1}$. 
We understand that the action of the operator $\partial_{r}$ or $\partial$ on a 
vector field ${\bm u}$ is associated with the function ${r}\mapsto {\bm u}({r}{\bm \sigma})$ for ${\bm \sigma}\in\mathbb{S}^{N-1}$ fixed, 
whereas $\nabla_{\!\sigma}$ or $\triangle_\sigma$ is associated with the function ${\bm \sigma}\mapsto {\bm u}({r}{\bm \sigma})$ for ${r}=|{\bm x}|$ fixed. 
As a simple example, the operation of $\nabla$ and $\triangle$ on the scalar field ${r}=|{\bm x}|$ or its powers gives 
$\nabla{r}={\bm \sigma}$ and $\triangle{r}^s=\alpha_s{r}^{s-2}$ for all $s\in\mathbb{R}$, where $\alpha_s$ is the same as in \eqref{alpha_s}.

\subsection{Radial-spherical-scalar representation of curl-free fields } 
Every vector field ${\bm u}\in C^\infty(\mathbb{R}^N)^N$ is said to be curl-free if 
\[
	\frac{\pd u_k}{\pd x_j}=\frac{\pd u_j}{\pd x_k} \quad \text{ on }\mathbb{R}^N\qquad\forall j,k \in \{1,\cdots,N \}, 
\]
or equivalently if there exists a scalar field $\phi\in C^\infty(\mathbb{R}^N)$ satisfying
\begin{equation}
\label{phi}
 \bu=\nabla \phi\quad\text{ on }\mathbb{R}^N.
\end{equation}
In view of this equation, we say that ${\bm u}$ has a scalar potential $\phi$. As another representation of curl-free fields, let us recall the following fact:
\begin{prop}[\cite{CF_Re}]
\label{prop:CF}
Let $\lambda\in\re$. 
Then a vector field ${\bm u}\in C^\infty(\re^N)^N$ is curl-free 
if and only if there exist two scalar fields $f,\varphi\in C^\infty(\dot{\mathbb{R}}^N)$ satisfying
\[
\left\{
\begin{array}{cl}
f \text{ is radially symmetric and } \ \int_{\mathbb{S}^{N-1}}\varphi(r {\bm \sigma}) \mathrm{d}\sigma=0&\forall\, r>0,\vspace{0.75em}
 \\
{r}^{1-\lambda}{\bm u}=  {\bm \sigma} \big(f+(\lambda+\pd_{t})\varphi\big) + \nabla_{\!\sigma} \varphi \quad&{\rm on}\ \re^N \setminus\{{\bm 0}\} .
\end{array}
\right.
\]
Moreover, such $f$ and $\varphi$ are uniquely determined,  and they are explicitly given by the equations 
\[
 f={r}^{-\lambda}\partial\overline{\phi}\quad\text{ and }\quad
 \varphi={r}^{-\lambda}\(\phi-\overline{\phi}\),
\]
where we set $\overline{\phi}({\bm x})=\frac{1}{|\mathbb{S}^{N-1}|}\int_{\mathbb{S}^{N-1}}\phi(|{\bm x}|{\bm \sigma})\mathrm{d}\sigma$ as the spherical mean of the scalar potential $\phi$ given by \eqref{phi}. In particular, if ${\bm u}$ has a compact support on $\dot{\mathbb{R}}^N$, then so do $f$ and $\varphi$.
\end{prop}
Later, we will use Proposition \ref{prop:CF} by choosing $\lambda = 2 - \frac{N}{2} - \gamma$.  (See \eqref{eta}).

\section{Proof of Theorem \ref{theorem:RH}}
\label{section3}

In this section, we prove Theorem \ref{theorem:RH}. 
Roughly speaking, the proof consists of theoretical part (\S\ref{subsec:compact_supp} and \S\ref{subsec:integrals}) and computational part (from \S\ref{subsec:Q0/P0} to \S\ref{subsec:opt}). 
Since the theoretical part is already well established in our previous work, we will only state its minimum required content: we exploit some $L^2$ formulae of curl-free fields given in \cite{CF_Re}.  
Instead, emphasis is placed on the computational part.

\subsection{Reduction to the case of compact support on $\dot{\mathbb{R}}^N$} 
\label{subsec:compact_supp}

Let $\phi$ be the scalar potential of the curl-free field ${\bm u}$ satisfying $\phi({\bm 0})=0$. 
The integrability condition ${\bm u}\in\mathcal{D}_{\gamma-1}(\mathbb{R}^N)$ \big(namely $\int_{\mathbb{R}^N}|{\bm u}|^2|{\bm x}|^{2\gamma-4}dx<\infty$\big) 
together with the smoothness of ${\bm u}$ on $\mathbb{R}^N$ implies that there exists an integer $m>\frac{3}{2}-\gamma$ satisfying
\[
 {\bm u}({\bm x})=\nabla\phi({\bm x})=O(|{\bm x}|^m),
\quad\text{ and hence }\quad
\left\{\begin{array}{l}
\phi({\bm x})=O(|{\bm x}|^{m+1}),
\\
 \nabla {\bm u}({\bm x})=O(|{\bm x}|^{m-1}),
\\
 \triangle {\bm u}({\bm x})=O(|{\bm x}|^{m-2})
\end{array}\right.
\]
as ${\bm x}\to {\bm 0}$. 
Then it additionally follows that the integrals 
\begin{equation}
\label{finite}
\int_{\mathbb{R}^N}\phi^2|{\bm x}|^{2\gamma-6}dx,\quad
\int_{\mathbb{R}^N}|\nabla {\bm u}|^2|{\bm x}|^{2\gamma-2}dx\quad\text{ and }\quad\int_{\mathbb{R}^N}|\triangle {\bm u}|^2|{\bm x}|^{2\gamma}dx
\end{equation}
are all finite. 

For the purpose of deriving the best constant $C_{N,\gamma}$ in inequality \eqref{RH_cf}, it is enough to consider the case where the curl-free field ${\bm u}=\nabla \phi$ is compactly supported on $\dot{\mathbb{R}}^N$. 
Here let us verify this fact. First of all let us define $\{{\bm u}_n\}\subset C_c^\infty(\dot{\mathbb{R}}^N)^N$ as a sequence of curl-free fields by 
\[
	 \bu_n(\bx):= \nabla \({\zeta}\(\frac{1}{n}\log|{\bm x}|\)\phi(\bx) \) \quad\text{for every }\ n\in\mathbb{N},
\]
where ${\zeta}\in C^\infty(\mathbb{R})$ such that 
${\zeta}({t})= \left\{\begin{array}{ll} 0&\text{ for }\ {t}\le-1\vspace{0.25em}\\	 
1&\text{ for }\ 1\le {t}\end{array}\right.$.
We use the abbreviations such as $\zeta_n=\zeta\(\frac{1}{n}\log|{\bm x}|\)$, $\zeta_n'=\zeta'\(\frac{1}{n}\log|{\bm x}|\)$, and $\zeta_n''=\zeta''\(\frac{1}{n}\log|{\bm x}|\)$. 
Noticing the asymptotic formulae
\begin{align*}
	\partial\zeta_n=\frac{1}{n}\zeta_n'=O(n^{-1}),\quad \nabla\zeta_n=\frac{{\bm \sigma}}{n{r}}\zeta_n'=\frac{{\bm \sigma}}{{r}}O(n^{-1}),
\quad \nabla\zeta_n'= \frac{{\bm \sigma}}{n{r}}\zeta_n''= \frac{{\bm \sigma}}{{r}}O(n^{-1})
\end{align*}
as $n\to\infty$,
we have the following calculations:
\begin{align*}
 \bu_n &=\nabla\(\zeta_n\phi\)=(\nabla\zeta_n)\phi+\zeta_n\nabla\phi
 = \frac{{\bm \sigma}\phi}{{r}} O(1/n)+ \zeta_n \bu   
, \\
 \nabla \bu_n &=
 \nabla\(\frac{{\bm \sigma}}{n\hspace{0.1em}{r}}{\zeta}'_n \phi+{\zeta}_n {\bm u}\)
 =\frac{{\bm \sigma}}{n{r}}(\nabla\zeta_n')\phi+\frac{\zeta_n'}{n}\nabla\frac{{\bm \sigma}\phi}{{r}}+(\nabla\zeta_n){\bm u}+\zeta_n\nabla {\bm u}
 \\&=\frac{{\bm \sigma}{\bm \sigma}\phi}{{r}^2} O(n^{-2})+\frac{\zeta_n'}{n}\(\frac{{\bm \sigma}}{{r}}\nabla \phi+\(\nabla \frac{{\bm \sigma}}{{r}}\)\phi\)+\(\frac{{\bm \sigma}}{{r}}{\bm u}\) O(n^{-1})+\zeta_n\nabla {\bm u}
 \\&=\frac{{\bm \sigma}{\bm \sigma}\phi}{{r}^2} O(n^{-1})+{\bm \sigma}\frac{{\bm u}}{{r}} O(n^{-1})+\zeta_n\nabla {\bm u}
\\ 
&\(
\begin{array}{l}
\text{where we abbreviate as }
{\bm v}{\bm w}:={\bm v}\otimes {\bm w}=\(v_iw_j\)_{i,j\in\{1,\cdots,N\}^2}
\\ 
\text{the tensor product of two vector fields}
\end{array}
\),
\\
\triangle {\bm u}_n &=\triangle\nabla\(\zeta_n\phi\)=\nabla\big((\triangle\zeta_n)\phi+2(\partial_{r}\zeta_n)\partial_{r}\phi+\zeta_n\triangle\phi\big)
 \\&=\nabla\(\(\partial^2\zeta_n+(N-2)\partial\zeta_n\){r}^{-2}\phi+\frac{2\partial\zeta_n}{{r}} \partial_{r}\phi\)+\nabla\(\zeta_n\triangle\phi\)
\\&=\nabla\(\(n^{-2}\zeta_n''+(N-2)n^{-1}\zeta_n'\)\frac{\phi}{{r}^2}+\frac{2\zeta_n'}{n} \frac{\partial_{r}\phi}{{r}}\)+\nabla\(\zeta_n\triangle\phi\)
 \\&=\Big(\underbrace{n^{-2}\nabla\zeta_n''+(N-2)n^{-1}\nabla\zeta_n'\strut}_{\frac{{\bm \sigma}}{{r}}O(n^{-1})}\Big)\frac{\phi}{{r}^2}
+\Big(\underbrace{n^{-2}\zeta_n''+(N-2)n^{-1}\zeta_n'\strut}_{O(n^{-1})}\Big)\nabla\frac{\phi}{{r}^2}
 \\&\quad+\frac{2\nabla\zeta_n'}{n}\frac{\partial_{r}\phi}{{r}}+\frac{2\zeta_n'}{n}\nabla \frac{\partial_{r}\phi}{{r}}+(\nabla\zeta_n)\triangle\phi+\zeta_n\triangle\nabla\phi
 \\&=\frac{{\bm \sigma}}{{r}} O(n^{-1})\frac{\phi}{{r}^2}+O(n^{-1})\nabla \frac{\phi}{{r}^2}+\frac{{\bm \sigma}\partial_{r}\phi}{{r}^2}O(n^{-2})+O(n^{-1})\nabla \frac{\partial_{r}\phi}{{r}}
 \\&\quad 
+{\bm \sigma}O(n^{-1})\frac{\triangle\phi}{{r}}+\zeta_n\triangle\nabla\phi
 \\&=\frac{{\bm \sigma}\phi}{{r}^3}O(n^{-1})+\frac{{\bm u}}{{r}^2}O(n^{-1})+\frac{{\bm \sigma}u_R}{{r}^2}O(n^{-1})
 +\frac{\partial_{r}{\bm u}}{{r}}O(n^{-1})
\\&\quad +\frac{{\bm \sigma}{\rm div}\hspace{0.1em}{\bm u}}{{r}} O(n^{-1})+\zeta_n\triangle {\bm u}
\end{align*}
hold as $n\to\infty$. 
Therefore, taking the $L^2(|{\bm x}|^{2\gamma}dx)$ integration yields
\begin{align*}
 &\int_{\re^N}\frac{|{\bm u}_n|^2}{|{\bm x}|^4}|{\bm x}|^{2\gamma}dx=\int_{\re^N}\frac{|\zeta_n{\bm u}|^2}{|{\bm x}|^4}|{\bm x}|^{2\gamma}dx+O(n^{-1})\to\int_{\mathbb{R}^N}\frac{|{\bm u}|^2}{|{\bm x}|^4}|{\bm x}|^{2\gamma}dx, \\
 &\int_{\re^N}\frac{|\nabla {\bm u}_n|^2}{|{\bm x}|^2}|{\bm x}|^{2\gamma}dx= \int_{\re^N}\frac{|\zeta_n\nabla {\bm u}|^2}{|{\bm x}|^2}|{\bm x}|^{2\gamma}dx+O(n^{-1})\to\int_{\mathbb{R}^N}\frac{|\nabla {\bm u}|^2}{|{\bm x}|^2}|{\bm x}|^{2\gamma}dx,
\\
 &\int_{\mathbb{R}^N}|\triangle {\bm u}_n|^2|{\bm x}|^{2\gamma}dx=\int_{\mathbb{R}^N}|\zeta_n\triangle {\bm u}|^2|{\bm x}|^{2\gamma}dx+O(n^{-1})\to\int_{\mathbb{R}^N}|\triangle {\bm u}|^2|{\bm x}|^{2\gamma}dx
\end{align*}
with the aid of the integrability conditions \eqref{finite}. This fact shows that the two integrals in \eqref{RH_cf} can be approximated by curl-free fields with compact support on $\dot{\mathbb{R}}^N$, as desired.

\subsection{Radial- and spherical-scalar expression of the integrals} 
\label{subsec:integrals}

In the rest of the present section, we use the notation \[{t}=\log|{\bm x}|=\log{r}\]
for an alternative radial coordinate obeying the differential rules
 \begin{equation}
\label{Emden} \partial_t ={r}\partial_{r}=\partial={\bm x}\cdot\nabla \quad\text{ and }\quad d{t}=d{r}/{r},
\end{equation}
which reproduces the same notation $\partial$ given in \eqref{der}. 
For any parameter $\lambda\in\mathbb{R}$, let $f$ and $\varphi$ be the scalar fields determined by the curl-free field ${\bm u}$, as given in Proposition \ref{prop:CF}, and we set 
\begin{equation}
\label{BV} 
	{\bm v}({\bm x})=|{\bm x}|^{1-\lambda}{\bm u}({\bm x})
\end{equation} 
as a new vector field in $C_c^\infty(\dot{\mathbb{R}}^N)^N$. 
Then the equation
\begin{equation}
\label{eq:v}
	 {\bm v} 
	={\bm \sigma} \big(f+(\lambda+\pd)\varphi\big) + \nabla_{\!\sigma} \varphi
\end{equation}
holds on $\dot{\mathbb{R}}^N$. Here we keep in mind that the equations \eqref{BV} and \eqref{eq:v} are invariant under the following replacement of the quadruple:
\begin{equation}
 \label{replace} (f,\varphi,{\bm v},{\bm u})\longmapsto\(\partial f,\partial\varphi,\partial {\bm v}, {r}^{\lambda-1}\partial({r}^{1-\lambda}{\bm u})\).
\end{equation}
Now, we choose 
\begin{equation}
\label{eta}
\lambda=2-\frac{N}{2}-\gamma,
\end{equation}
and  let us recall from \cite[\S3.2]{CF_Re} that the integral on the right-hand side of the Hardy-Leray inequality \eqref{HL} can be expressed in terms of $({\bm v},f,\varphi)$ as follows:
\begin{align*}
\tl{L2_Du}
 \mathop{\int}_{\re^N}|\nabla {\bm u}|^2|{\bm x}|^{2\gamma}dx  
	&=\mathop{\iint}_{\re\times\Sp^{N-1}}\((\lambda-1)^2|{\bm v}|^2+|\pd{\bm v}|^2+|\nabla_{\!\sigma} {\bm v}|^2\)dt\hspace{0.1em}\mathrm{d}\sigma,
\\ 
\mathop{\iint}_{\re\times\Sp^{N-1}}|\nabla_{\!\sigma} {\bm v}|^2dt\hspace{0.1em}\mathrm{d}\sigma
&=\mathop{\iint}_{\re \times \Sp^{N-1}}\Big( (\triangle_\sigma\varphi)^2 + \((\lambda-2)^2-2N\) |\nabla_{\!\sigma} \varphi|^2 \Big) dt \hspace{0.1em}\mathrm{d}\sigma  
\\  &\quad
+ \iint_{\re \times \Sp^{N-1}}\( |\pd \nabla_{\!\sigma} \varphi|^2 + (N-1)|\bv|^2  \) dt \hspace{0.1em}\mathrm{d}\sigma,\tl{L2_Ds_v} 
\\
\mathop{\iint}_{\re \times \Sp^{N-1}}|{\bm v}|^2 dt \hspace{0.1em}\mathrm{d}\sigma
&=\iint_{\re \times \Sp^{N-1}}\(f^2+(\pd \varphi)^2+\lambda^2\varphi^2+|\nabla_{\!\sigma}\varphi|^2\)dt \hspace{0.1em}\mathrm{d}\sigma
\\&=\iint_{\mathbb{R}\times\mathbb{S}^{N-1}}\(f^2+\varphi\(\lambda^2-\partial^2-\triangle_\sigma\)\varphi\)d{t}\mathrm{d}\sigma. \tl{L2_v} 
\end{align*}
Here the last equality follows from integration by parts together with the support compactness of ${\bm v}$ or $f,\varphi$.  
Applying \eqref{L2_v} to \eqref{replace}, we also obtain
\begin{align*}
 \mathop{\iint}_{\mathbb{R}\times\mathbb{S}^{N-1}}|\partial{\bm v}|^2d{t}\hspace{0.1em}\mathrm{d}\sigma&=\mathop{\iint}_{\re \times \Sp^{N-1}}\Big((\partial f)^2+(\partial^2\varphi)^2+\lambda^2(\partial\varphi)^2+|\partial\nabla_{\!\sigma}\varphi|^2\Big)dt \hspace{0.1em}\mathrm{d}\sigma
\\&=\mathop{\iint}_{\re \times \Sp^{N-1}}\Big(f(-\partial^2)f+\varphi\(\lambda^2-\partial^2-\triangle_\sigma\)(-\partial^2\varphi)\Big)dt \hspace{0.1em}\mathrm{d}\sigma
\tl{L2_dv}
\end{align*}
by integration by parts. 
After plugging \eqref{L2_Ds_v} into \eqref{L2_Du}, substitute \eqref{L2_v} and \eqref{L2_dv} into the $L^2$ terms of ${\bm v}$ and $\partial {\bm v}$; then we get 
\begin{align*}
 \mathop{\int}_{\mathbb{R}^N}|\nabla {\bm u}|^2|{\bm x}|^{2\gamma}dx&=\((\lambda-1)^2+N-1\)\mathop{\iint}_{\mathbb{R}\times\mathbb{S}^{N-1}}|{\bm v}|^2d{t}\mathrm{d}\sigma+\mathop{\iint}_{\mathbb{R}\times\mathbb{S}^{N-1}}|\partial {\bm v}|^2d{t}\mathrm{d}\sigma
 \\ &\quad+\mathop{\iint}_{\mathbb{R}\times\mathbb{S}^{N-1}}\Big((\triangle_\sigma\varphi)^2 + \((\lambda-2)^2-2N\) |\nabla_{\!\sigma} \varphi|^2+|\partial\nabla_{\!\sigma}\varphi|^2\Big)d{t}\mathrm{d}\sigma
\\ &=\((\lambda-1)^2+N-1\)\mathop{\iint}_{\mathbb{R}\times\mathbb{S}^{N-1}}\(f^2+\varphi\(\lambda^2-\partial^2-\triangle_\sigma\)\varphi\) d{t}\mathrm{d}\sigma
\\&\quad +\iint_{\mathbb{R}\times\mathbb{S}^{N-1}}\Big(f(-\partial^2)f+\varphi\(\lambda^2-\partial^2-\triangle_\sigma\)(-\partial^2\varphi)\Big)d{t}\mathrm{d}\sigma
 \\ &\quad+\iint_{\mathbb{R}\times\mathbb{S}^{N-1}}\varphi\Big(\triangle_\sigma^2 - \((\lambda-2)^2-2N\)\triangle_\sigma +\partial^2\triangle_\sigma\Big)\varphi d{t}\mathrm{d}\sigma
\\&=\iint_{\mathbb{R}\times\mathbb{S}^{N-1}}\(\varphi \mathcal{P}_1(-\partial^2,-\triangle_\sigma,\lambda)\varphi+f\mathcal{P}_0(-\partial^2,\lambda)f\)d{t}\mathrm{d}\sigma
\end{align*}
by integration by parts, where we have defined two polynomials $\mathcal{P}_1$ and $\mathcal{P}_0$ by 
\begin{equation*}
\left\{\begin{split}
	 \mathcal{P}_1(\tau, {a},\lambda)&=\((\lambda-1)^2+N-1\)\(\lambda^2+\tau+{a}\)
	 \\&\quad +\(\lambda^2+\tau+a\)\tau+a^2+\((\lambda-2)^2-2N\){a}+{a}\tau
	 \\&={a}^2+\(2\lambda^2-6\lambda+4-N+2\tau\){a}
	 \\&\quad +\(\lambda^2+\tau\)\(\(\lambda-1\)^2+N-1+\tau\),
	 \\ \mathcal{P}_0(\tau,\lambda)&=(\lambda-1)^2+N-1+\tau.
	\end{split}
 \right.
\end{equation*}
Now let us replace $\gamma$ by $\gamma-1$; in view of \eqref{eta}, this manipulation is equivalent to replacing $\lambda$ by $\lambda+1$. Then the result of the above integral computation changes into
\begin{equation}
\label{poly:P}
\left.
\begin{aligned}
 \mathop{\int}_{\mathbb{R}^N}\frac{|\nabla {\bm u}|^2}{|{\bm x}|^2}|{\bm x}|^{2\gamma}dx
&=\mathop{\iint}_{\mathbb{R}\times\mathbb{S}^{N-1}}\(\varphi P_1(-\partial^2,-\triangle_\sigma)\varphi+fP_0(-\partial^2)f\)d{t}\mathrm{d}\sigma,\\
\text{where and hereafter}&\text{ we abbreviate as}
\\
P_1(\tau,{a}):\hspace{-0.25em}&=\mathcal{P}_1(\tau,{a},\lambda+1)
\\&={a}^2+\Big(2\(\lambda^2-\lambda+\tau\)-N\Big){a}
\\&\quad  +\((\lambda+1)^2+\tau\)\(\lambda^2+N-1+\tau\),
\\ P_0(\tau):\hspace{-0.25em}&=\mathcal{P}_0(\tau,\lambda+1)=\lambda^2+N-1+\tau,
\end{aligned}
\right\}
\end{equation}
as the expression in terms of $f,\varphi$ for the integral on the right-hand side of the Rellich-Hardy inequality \eqref{RH_cf}. 
To express the left-hand side,  
we exploit the result of \cite[Eq.(30),(31) with $\lambda$ replaced by $\lambda+1$]{CF_Re}: it holds that
\begin{equation}
\hspace{-0.75em}
\label{poly:Q}
\left.\begin{aligned}
 \mathop{\int}_{\re^N}|\triangle \bu|^2|\bx|^{2\gamma}&dx
=\mathop{\iint}_{\re\times\Sp^{N-1}}\Big(\varphi\,Q_1(-\pd^2,-\triangle_\sigma)\varphi + f Q_0(-\pd^2)f\Big) dt\hspace{0.1em}{\mathrm d} \sigma, 
\\ 
\text{where $Q_1$ and}& \text{ $Q_0$ are the polynomials given by}\\
Q_1(\tau,{a}) 
	& 
	= \scalebox{0.89}[1]{$\(\tau+{a}+(\lambda-1)^2\) 
	\(\begin{array}{l}
	    	 \(\tau+{a}+(\lambda+1)^2\)\(\tau+{a}+(\lambda+N-1)^2\)
	   \vspace{0.5em}
	  \\ -(2\lambda+N)^2{a}
	  \end{array}   
\)$}	
\\&=\scalebox{0.85}[1]{$\big(\tau+{a}+(\lambda-1)^2\big)\Big(\tau^2+\(2\({a}+\alpha_{\lambda+1}\)+(N-2)^2\)\tau+\({a}-\alpha_{\lambda+1}\)^2\Big)$}, 
 \\
 Q_0(\tau) &=\(\tau+(\lambda-1)^2\)\(\tau+(\lambda+N-1)^2\). 
\end{aligned}
\right\}\hspace{-0.75em}
\end{equation}
To proceed further, let us apply to $\varphi$ and $f$ the one-dimensional Fourier transformation with respect to $t$: 
we set
\[
\widehat{\varphi}(\tau,{\bm \sigma})= \frac{1}{\sqrt{2\pi}}\int_{\re}e^{-i\tau t}\varphi(e^{t}{\bm \sigma})dt,\quad\
	\widehat{f}(\tau)= \frac{1}{\sqrt{2\pi}}\int_{\re}e^{-i\tau t}f(e^t {\bm \sigma})dt
\]
for $(\tau,{\bm \sigma})\in\re\times\Sp^{N-1}$, where $i=\sqrt{-1}$.  
Also we  apply to $\widehat{\varphi}$ the spherical harmonics decomposition:
\[
\widehat{\varphi}
=\sum_{\nu\in\mathbb{N}}\widehat{\varphi}_\nu
,
	\qquad
 \left\{\begin{array}{l}
	-\triangle_\sigma \widehat{\varphi}_\nu=\alpha_\nu \widehat{\varphi}_\nu,\vspace{0.5em} \\ \alpha_\nu=\nu(\nu+N-2)\quad\forall\nu\in\mathbb{N}.
\end{array}\right.
\]
Now, we are in a position to evaluate the quantity 
\begin{equation}
\label{R-H quotient}
	\frac{\int_{\mathbb{R}^N}|\triangle {\bm u}|^2|{\bm x}|^{2\gamma}dx}{\int_{\mathbb{R}^N}|\nabla {\bm u}|^2|{\bm x}|^{2\gamma-2}dx},
\end{equation}
which we simply call the {\it R-H quotient}.
To this end, by using \eqref{poly:P}, \eqref{poly:Q} and the $L^2(\re)$ isometry of the Fourier transformation, we have
\begin{align*}
 \frac{\int_{\re^N}|\triangle {\bm u}|^2|{\bm x}|^{2\gamma}dx}{\int_{\re^N}|\nabla{\bm u}|^2|{\bm x}|^{2\gamma-2}dx}
 &=\frac{\disp \iint_{\re\times\mathbb{S}^{N-1}}\(\sum_{\nu \in \N} Q_1(\tau^2,\alpha_\nu)|\widehat{\varphi_\nu}|^2 + Q_0(\tau^2)|\widehat{f}|^2\)d\tau\hspace{0.1em}\mathrm{d}\sigma}{\disp \iint_{\re\times\mathbb{S}^{N-1}}\(\sum_{\nu \in \N} P_1(\tau^2,\alpha_\nu)|\widehat{\varphi_\nu}|^2 + P_0(\tau^2)|\widehat{f}|^2\)d\tau\hspace{0.1em}\mathrm{d}\sigma}
\\
 &\ge \min\left\{\inf_{\tau\in\re\backslash\{0\}}\frac{Q_0(\tau^2)}{P_0(\tau^2)},\ \inf_{\nu\in\N}\inf_{\tau\in\re\backslash\{0\}}\frac{Q_1(\tau^2,\alpha_\nu)}{P_1(\tau^2,\alpha_\nu)}\right\}  
\\&= \min\left\{\inf_{\tau>0}\frac{Q_0(\tau)}{P_0(\tau)},\ \inf_{\nu\in\N}\inf_{\tau>0}\frac{Q_1(\tau,\alpha_\nu)}{P_1(\tau,\alpha_\nu)}\right\}.  
\tl{quotient:RH}
\end{align*}
Hence, our goal is reduced to evaluate the fractions $Q_0/P_0$ and $Q_1/P_1$.
In the following subsections, we will show that the infimum values of these fractions are achieved at $\tau=0$.

\subsection{Evaluation of $Q_0/P_0$}
\label{subsec:Q0/P0}

A direct calculation yields
\[
 \begin{aligned}
  \frac{Q_0(\tau)}{P_0(\tau)}&=\frac{\(\tau+(\lambda-1)^2\)\(\tau+(\lambda+N-1)^2\)}{\lambda^2+N-1+\tau}
  \\&=\tau+\(\lambda+N-2\)^2+(N-1)\(1-\frac{\(2\lambda+N-2\)^2}{\tau+\lambda^2+N-1}\)
 \end{aligned}
\]
for all $\tau\ge0$.  
The last expression is of the form $g(\tau) = \tau + a - \frac{b}{\tau + c}$ for some constants $a,b\ge0$ and $c > 0$, which leads to $g(\tau)-g(0)=\tau+ \frac{b\tau}{c(\tau+c)}\ge\tau$. 
Thus we have 
\begin{equation}
	\label{est:Q0/P0}  \frac{1}{\tau}\(\frac{Q_0(\tau)}{P_0(\tau)}-\frac{Q_0(0)}{P_0(0)}\)\ge 1\qquad\forall\tau>0,
\end{equation}
whence in particular we obtain
$\quad \displaystyle\inf_{\tau>0}\frac{Q_0(\tau)}{P_0(\tau)} =\frac{Q_0(0)}{P_0(0)}=\frac{(\lambda-1)^2\(\lambda+N-1\)^2}{\lambda^2+N-1}$.

\subsection{The case when $P_1$ has zeros} 
\label{subsec:P1_zero}
Here we specify when $P_1(\tau,\alpha_\nu)=0$ happens. 
Notice from \eqref{poly:P} that $P_1(\tau,{a})$ is strictly monotone increasing in $\tau\ge0$ for any ${a}>0$, and hence it holds that
\[
\begin{split}
  P_1(\tau,{a})
&>P_1(0,{a})
 \\&={a}^2+\(2(\lambda-1/2)^2- \tfrac{1}{2}-N\){a}+(\lambda+1)^2\(\lambda^2+N-1\)
\end{split}
\]
for all $\tau>0$, as well as that 
\[
 P_1(\tau,\alpha_1)>P_1(0,\alpha_1)
=P_1(0,N-1)
=\lambda^2\((\lambda+1)^2+3(N-1)\)
\]
for all $\tau>0$. Notice on the right-hand side of the (three lines) above inequality that the center of the graph of the quadratic function ${a}\mapsto P_1(0,{a})$ is located at ${a}=-\(\lambda-\frac{1}{2}\)^2+\frac{1}{4}+\frac{1}{2}N\le 2N=\alpha_2$. Then we see that for all  $\tau>0$ and $\nu\ge2$, 
\[
\begin{split}
   P_1(\tau,\alpha_\nu)&>P_1(0,\alpha_\nu)
\\&\ge P_1(0,\alpha_2)
\\&=\lambda^4+2 \lambda^3+5N \lambda^2-2(N+1)\lambda+(N+1)(2N-1)
\\&=\lambda^2(\lambda+1)^2+2(2N-1)\lambda^2+(N+1)(\lambda-1)^2+2(N^2-1)
\\&\ge2(N^2-1)>0.
\end{split}
\]
In view of the above discussion, we see that 
\[
  P_1(\tau,\alpha_\nu)=0 \quad \text{ holds if and only if }\ (\tau,\nu,\lambda)=(0,1,0)
\] 
Hence, every time we treat the rational polynomial $Q_1/P_1$,   we have to deal with the case $\lambda=0$ \big(or equivalently $\gamma=2-\frac{N}{2}$\big) as a special one. 
For this reason, in the rest of this paper  we always assume $\lambda\ne0$ \big($\Leftrightarrow\gamma\ne2-\frac{N}{2}$\big) unless others are specified.

\subsection{Evaluation of $Q_1/P_1$}
\label{subsec:Q1/P1}

Let us check that
\[
\inf_{\tau>0} \frac{Q_1(\tau,{a})}{P_1(\tau,{a})}=\frac{Q_1(0,{a})}{P_1(0,{a})}
\qquad \forall{a}\in\{\alpha_\nu\}_{\nu\in\mathbb{N}}
\]
in order to evaluate \eqref{quotient:RH} from below. 
This equation is equivalent to the inequality $\frac{Q_1(\tau,{a})}{P_1(\tau,{a})}\ge \frac{Q_1(0,{a})}{P_1(0,{a})}$ $(\forall\tau>0)$, 
which can be verified by directly evaluating $Q_1(\tau,{a})P_1(0,{a})-P_1(\tau,{a})Q_1(0,{a})$ to be nonnegative. 
However, we further show the following stronger fact, which serves as a key tool for the proof of Theorem \ref{theorem:RH_c}: 

\begin{lemma}
\label{lemma:Q/P}
There exists a constant number ${c_0}>0$ such that the inequality
 \[\frac{1}{\tau}\(\frac{Q_1(\tau,{a})}{P_1(\tau,{a})}-\frac{Q_1(0,{a})}{P_1(0,{a})}\)\ge {c_0}\] 
holds for all $\tau>0$ and $a\in\{\alpha_\nu\}_{\nu\in\mathbb{N}}$. 
\end{lemma}
Here we  give the proof of the lemma only for the case $\gamma\le1$. Since the proof for $\gamma>1$ follows by a similar technique, we postpone it in later section (see \S \ref{sec:Q/P}). 
The \tcr{proof of Lemma \ref{lemma:Q/P}} consists of tedious computations, and we used {\it Maxima} \tcr{in the course of the proof.} 
\tcr{However, we need many computational techniques to simplify the calculations and ideas to make the proof understandable, even with the use of Maxima.}

\noindent
{\it Proof of Lemma \ref{lemma:Q/P} for $\gamma\le 1$ \big(or equivalently $\lambda\ge 1-\frac{N}{2}$\big).}  
It suffices to check the inequality for $c_0=1$:
\[\frac{1}{\tau}\(\frac{Q_1(\tau,{a})}{P_1(\tau,{a})}-\frac{Q_1(0,{a})}{P_1(0,{a})}\)\ge 1\qquad\forall\tau>0,\quad\forall{a}\in\left\{\alpha_\nu\right\}_{\nu\in\mathbb{N}}.\] 
To  this end,  we directly compute the left-hand side minus right-hand side:   
 by using \eqref{poly:P} and \eqref{poly:Q} we get
\begin{align*}
 \frac{1}{\tau}&\(\frac{Q_1(\tau,{a})}{P_1(\tau,{a})}-\frac{Q_1(0,{a})}{P_1(0,{a})}\)-1
 \\&=\frac{1}{\tau}\(\frac{\(\tau+{a}+(\lambda-1)^2\) \Big(\tau^2+\(2\({a}+\alpha_{\lambda+1}\)+(N-2)^2\)\tau+\({a}-\alpha_{\lambda+1}\)^2\Big)
 }{{a}^2+\big(2\(\lambda^2-\lambda+\tau\)-N\big){a}
 +\((\lambda+1)^2+\tau\)\(\lambda^2+N-1+\tau\)}\right.\hspace{-1.5em}
\\&\qquad\quad\  \left.-\ \frac{\({a}+(\lambda-1)^2\) 
 \({a}-\alpha_{\lambda+1} \)^2 
 }{{a}^2+\big(2(\lambda^2-\lambda)-N\big){a}
 +(\lambda+1)^2\(\lambda^2+N-1\)}
 \)-1
 \\&=\ \underbrace{\!\!(2 \lambda+N-2)_{\vphantom{{\substack{A\\ Aa}}}}\hspace{-0.5em}\strut}_{\quad \ge\,0}\  \frac{ \ G_0({a})+G_1({a}) \tau}{P_1(0,{a})P_1(\tau,{a})},
\tl{q/p:1}  
\end{align*}
where we have defined
\begin{equation}
\label{G0G1}
 \left.\begin{split}
	 G_0({a})&:= (2\lambda+N){a}^3+\Big(\(2\lambda^2-N+5\)(2\lambda+N)-2(N-1)\Big){a}^2
	 \\&\quad\ +\(\begin{array}{l}
	 2\lambda^5+(N-8)\lambda^4-8N\lambda^3-2(N^2+2N-2)\lambda^2
\vspace{0.25em}\\ -\,2(6N-7)\lambda-2N^2-N+4
\end{array}
	 \){a}
	 \\&\quad \ +(N-1)(2\lambda+N-2)(\lambda+1)^4,
	 \\
	 G_1({a})&:=(2\lambda+N){a}^2+\Big((2\lambda+N)\((\lambda-1)^2-N+1\)-2(N-1)\Big){a}
	 \\&\quad\ +(N-1)(2\lambda+N-2)(\lambda+1)^2
	\end{split}\right\}
\end{equation}
as cubic and quadratic polynomials in ${a}$. Then the necessary and sufficient condition for the nonnegativity of \eqref{q/p:1} $(\forall\tau\ge0)$ is given by the inequalities 
\begin{equation*}
\label{G1G0}
G_1({a})\ge0\quad\text{ and }\quad G_0({a})\ge0
\qquad\forall{a}\in\{\alpha_\nu\}_{\nu\in\mathbb{N}},
\end{equation*}
whence our goal is reduced to showing them.  The first inequality is easier to prove, by considering the Taylor series of $G_1({a})$ at ${a}=\alpha_1$: a straightforward calculation yields 
\begin{align*}
G_1(\alpha_1+{s})
 & =s^2(2\lambda+N) +s\Big((N-1)(N+2\lambda-2)+(\lambda-1)^2(2\lambda+N)\Big)\\&\quad +2(N-1)\lambda^2(2\lambda+N-1)
\tl{Taylor:G1}
\end{align*}
for all ${s}\in\mathbb{R}$. 
Since $\lambda\ge1-\frac{N}{2}$, notice here that the coefficients of the powers of ${s}$ are all nonnegative, which tells us that $G_1(\alpha_1+{s})\ge0$ for all ${s}\ge0$. This fact directly implies $G_1({a})\ge0$ for all ${a}\in\{\alpha_\nu\}_{\nu\in\mathbb{N}}$, as desired.

Now, all we have to do is to show $G_0({a})\ge0$ for ${a}\in\{\alpha_\nu\}_{\nu\in\mathbb{N}}$. To do so, let us consider the Taylor series of $G_0({a})$ at ${a}=\alpha_1$, and we get
\begin{align*}
\tl{Taylor:G0} G_0(\alpha_1+{s})
 &=s^3(2\lambda+N)
 +s^2\mathcal{G}_2(\lambda)
 +s\mathcal{G}_1(\lambda)
  +2(N-1)\lambda^4(2\lambda+N-1)
\end{align*}
by a straightforward calculation, where
\begin{align*}
\tl{G2}
 \mathcal{G}_2(\lambda)&:= (2\lambda+N-2)\Big((\lambda+1)^2+\lambda^2+N+3\Big)+(N-1)^2+9,
  \\
\tl{G1} \mathcal{G}_1(\lambda)&
 :=\lambda^4(2\lambda+N-8)+2\lambda^2\(2-4\lambda-4N+N^2\)
+N^2(2\lambda+N).
\end{align*}
Noticing that $\mathcal{G}_2(\lambda)\ge0$ and $G_0(\alpha_1)\ge0$, we aim to prove the following fact:
 \begin{alignat}{3}
\label{case1}  &  \text{if}\quad &1-\tfrac{N}{2}\le\lambda\le 1 &\quad\text{ then}\quad \mathcal{G}_1(\lambda)\ge0,  
  \\
\label{case2}  \text{or }&\text{if}\quad &1<\lambda \quad\ \ \, 
  &\quad \text{ then}\quad G_0(\alpha_2+{s})\ge0\quad\forall{s}\ge0,
 \end{alignat}
which implies the desired inequality $G_0({a})\ge0$ \ $\forall{a}\in\{\alpha_\nu\}_{\nu\in\mathbb{N}}$.

For the proof of \eqref{case1}, let $\lambda$ be parameterized as
\[\lambda=
1-\frac{N{s}}{2},\qquad 0\le{s}\le1.\]
Then we directly compute
\[
 \begin{split}
\mathcal{G}_1(\lambda)
&=\(1-\frac{N{s}}{2}\)^4\(N-N{s}-6\)+2\(1-\frac{N{s}}{2}\)^2\(2N{s}-2-4N+N^2\)
\\&\quad +N^2\(N-N{s}+2\)
  \\&=\tfrac{1}{16}N^5(1-{s}){s}^4+\tfrac{1}{8}N^4({s}-2)^2{s}^2+\tfrac{1}{2}N^3(1-{s})\(2-4{s}-5{s}^2\)\\&\quad +2N^2(2+3{s}-6{s}^2)+N(19{s}-7)-10
  \\&=\tfrac{1}{16}(N-2)^5(1-s)s^4
  +\tfrac{1}{8}(N-2)^4s^2\Big({s}^2+(1-{s})\(5{s}^2+4\)\Big)
  \\&\quad +\tfrac{1}{2}(N-2)^3\Big((1-{s}){s}^2\((1-{s})^2+4{s}^2\)+2\(1-3{s}+3{s}^2\)\Big)
\\&\quad 
+(N-2)^2\Big((1-{s})^2\(4{s}+2(1-{s}^2)+5(1-{s}^3)\)+3-2{s}\Big)
\\&\quad +(N-2)\Big((1-s)^2s(17-s-5s^2)+21-10s-3s^2\Big)
 \\&\quad +2s\Big((1-s)(3+s)\(5-4s+s^2\)+4\Big)
 \end{split}
\]
as a Taylor series of the function $N\mapsto \mathcal{G}_1(\lambda)=\mathcal{G}_1\(1-\frac{N{s}}{2}\)$ at $N=2$. Notice here that the coefficients of the powers of $N-2$ are all nonnegative since $0\le {s}\le1$. Therefore, we get $\mathcal{G}_1(\lambda)\ge0$, as desired.

Now, all that is left is to show \eqref{case2}. To this end, notice from \eqref{Taylor:G0} that
\[
\begin{split}
 \frac{1}{{s}}\,  G_0(\alpha_1+{s})
&\ge {s}\hspace{0.1em}\mathcal{G}_2(\lambda)+\mathcal{G}_1(\lambda)
 \\&\ge {s} \hspace{0.1em}(2\lambda+N-2)\Big((\lambda+1)^2+\lambda^2\Big)
 \\&\quad  +	      \lambda^4(2\lambda+N-8)+2\lambda^2\(2-4\lambda-4N+N^2\)
\end{split}
\]
holds for all ${s}>0$. 
Replacing ${s}$ by $N+1+{s}$ on both sides, we then get
\[
 \begin{split}
  \frac{G_0(\alpha_2+{s})}{N+1+{s}}&=\frac{1}{N+1+{s}}\hspace{0.1em}G_0(\alpha_1+N+1+{s})
  \\&\ge (N+1+{s})(2\lambda+N-2)\((\lambda+1)^2+\lambda^2\)
  \\&\quad +\lambda^4(2\lambda+N-8)+2\lambda^2\(2-4\lambda-4N+N^2\)
\\&=2(\lambda-1)^5+(\lambda-1)^4(N+2)+4(\lambda-1)^3(s+2N-4)
\\&\quad +2(\lambda-1)^2\((N+6)s+2N^2+6N-18\)
\\&\quad +2(\lambda-1)\((3N+5)s+5N^2+2N-14\)
\\&\quad +5Ns+7N^2-2N-10
 \end{split}
\]
for all ${s}\ge0$. Notice here that the coefficients of the powers of $\lambda-1$ are all nonnegative since $N\ge2$. Therefore, from the assumption $\lambda>1$ we have obtained $G_0(\alpha_2+{s})\ge0$, as desired.
\qed

Since the polynomial function $P_1(\tau,{a})$ is quadratic in $\tau$,  it is clear from \eqref{q/p:1} that
\[
\lim_{\tau\to\infty} \frac{1}{\tau}\(\frac{Q_1(\tau,{a})}{P_1(\tau,{a})}-\frac{Q_1(0,{a})}{P_1(0,{a})}\)=1
\]
for each ${a}\ge\alpha_1$. Therefore, the constant number $c_0$ of Lemma \ref{lemma:Q/P} is optimal when $c_0=1$,   in the sense that
\[
\inf_{\tau>0}\inf_{\nu\in\mathbb{N}}\frac{1}{\tau}\(\frac{Q_1(\tau,\alpha_\nu)}{P_1(\tau,\alpha_\nu)}-\frac{Q_1(0,\alpha_\nu)}{P_1(0,\alpha_\nu)}\)=1
\]
holds as far as $\gamma\le1$. 

\subsection{A lower bound for the R-H quotient
}
\label{subsec:lower}

In view of the estimate \eqref{quotient:RH} for the R-H quotient \eqref{R-H quotient}, it follows from \S\ref{subsec:Q0/P0}, \S\ref{subsec:P1_zero} and \S\ref{subsec:Q1/P1} that the inequality
\[
\int_{\mathbb{R}^N}\left|\triangle {\bm u}\right|^2|{\bm x}|^{2\gamma}dx\ge C_{N,\gamma}\int_{\mathbb{R}^N}|\nabla{\bm u}|^2|{\bm x}|^{2\gamma-2}dx
\] 
holds for curl-free fields ${\bm u}$ with the constant number
 \[\begin{split}
  C_{N,\gamma}&=\min\left\{\inf_{\nu\in\N}\inf_{\tau>0}\frac{Q_1(\tau,\alpha_\nu)}{P_1(\tau,\alpha_\nu)},\ \inf_{\tau>0}\frac{Q_0(\tau)}{P_0(\tau)}\right\}
  \\&=\min\left\{\min_{\nu\in\N}\frac{Q_1(0,\alpha_\nu)}{P_1(0,\alpha_\nu)},\ \frac{Q_0(0)}{P_0(0)}\right\}.
 \end{split}\]
Notice from \eqref{poly:P} and \eqref{poly:Q} that the last two fractions are explicitly written  as
\[
 \begin{split}
\frac{Q_1(0,\alpha_\nu)}{P_1(0,\alpha_\nu)}
&=\frac{
\(\alpha_\nu+(\lambda-1)^2\)(\alpha_{\lambda+1}-\alpha_\nu)^2}
{\(\alpha_{\lambda+1}-\alpha_\nu\)^2+\(2\lambda+N-2\)\big((2\lambda+1)\alpha_\nu-(N-1)(\lambda+1)^2\big)}
  \\&=\tfrac{\((\gamma-2)^2-\(\nu+\frac{N}{2}-1\)^2\)^2\(\alpha_\nu+\(\gamma+\frac{N}{2}-1\)^2\)}{\((\gamma-2)^2-\(\nu+\frac{N}{2}-1\)^2\)^2+2(\gamma-1)\(\(2\gamma+N-5\)\alpha_\nu+(N-1)\(\gamma+\frac{N}{2}-3\)^2\)},
\\ 
  \frac{Q_0(0)}{P_0(0)}&=\frac{(\lambda-1)^2(\lambda+N-1)^2}{\lambda^2+N-1}=\frac{\((\gamma-1)^2-\frac{1}{4}N^2\)^2}{\(\gamma+\frac{N}{2}-2\)^2+N-1},
 \end{split}
\]
by recalling the notation \eqref{eta} together with the aid of the identity
\begin{equation}
\alpha_{s}-\alpha_{t}=\({s}+\tfrac{N}{2}-1\)^2-\({t}+\tfrac{N}{2}-1\)^2
\qquad\forall{s},{t}\in\mathbb{R};
\label{as-at}
\end{equation}
 in other words, we have
\[
 \begin{split}
\frac{Q_1(0,\alpha_\nu)}{P_1(0,\alpha_\nu)}=C_{N,\gamma,\nu},\qquad
\frac{Q_0(0)}{P_0(0)}=C_{N,\gamma,0}
 \end{split}
\]
in terms of the same notation in Theorem \ref{theorem:RH}.  
Therefore, we have obtained 
\[
 C_{N,\gamma}=\min_{\nu\in\mathbb{N}\cup\{0\}}C_{N,\gamma,\nu}
\]
as a lower bound for the R-H quotient \eqref{R-H quotient}, which coincides with the same constant number $C_{N,\gamma}$ given in Theorem \ref{theorem:RH}.

\subsection{Sharpness of $C_{N,\gamma}$}
\label{subsec:opt}

We show here the optimality for the constant $C_{N,\gamma}$ in the inequality \eqref{RH_cf}. To this end, we construct a sequence of curl-free fields minimizing the value of the R-H quotient \eqref{R-H quotient}.
First of all,  choose $\nu_0 \in \N\cup\{0\}$ to be such that
\[
 \min_{\nu\in\mathbb{N}\cup\{0\}}C_{N,\gamma,\nu}=C_{N,\gamma,\nu_0}.
\]
If $\gamma=2-\frac{N}{2}$, by the same computation as \eqref{C_2-n/2} below (\S \ref{sec:observation}) we have $C_{N,\gamma,1}>C_{N,\gamma,0}$ for all $N\ge2$, which implies that $\nu_0\ne1$. 
Hence it follows from \S\ref{subsec:P1_zero} that we can always assume that $P_1(0,\alpha_{\nu_0})>0$. 

Define a sequence of curl-free fields $\{\bu_n\}_{n\in\N}\subset C_c^\infty(\dot{\re}^N)^N$ by the formula
\[
 {\bm u}_n({\bm x})
:=\begin{cases}
   \  {\bm x} |{\bm x}|^{\lambda-1}h\(\frac{1}{n}\log|{\bm x}|\)&\text{if }\nu_0=0\vspace{0.25em}
\\\ \nabla\(|{\bm x}|^{\lambda+1}\varphi_n({\bm x})\)&\text{otherwise}
\end{cases}
\]
together with
\[
 \varphi_n({\bm x})=h\(\tfrac{1}{n}\log|{\bm x}|\)Y({\bm x}/|{\bm x}|)
\]
for any $h\in C_c^\infty(\mathbb{R})\backslash\{0\}$ such that $\int_{\mathbb{R}}(h({t}))^2d{t}=1$. Here  $Y\in C^\infty(\mathbb{S}^{N-1})\setminus\{0\}$ denotes a spherical harmonic function satisfying the eigenequation \[-\triangle_\sigma Y=\alpha_{\nu_0}Y\qquad\text{on }\mathbb{S}^{N-1}.\] 
Also define $\{{\bm v}_n\}\subset C_c^\infty(\dot{\mathbb{R}}^N)^N$ as a sequence of vector fields by the formula  \[\bu_n({\bm x})=|{\bm x}|^\lambda{\bm v}_n({\bm x})\]
in the same way as \eqref{BV}. 
Then we have
\[
	\bv_n=\left\{
	\begin{array}{ll}
	{\bm \sigma}f_n & \text{ if }\ \nu_0=0
	\vspace{0.5em}\\ 
	{\bm \sigma}(\pd +\lambda+1)\varphi_n+\nabla_{\!\sigma}\varphi_n &\text{ otherwise}
	\end{array}\right.,
\]
where $f_n$ is given by $f_n({\bm x})=h(\frac{1}{n}\log|{\bm x}|)$. 
In this setting, let us now apply the formulae \eqref{poly:P} and \eqref{poly:Q} to the case $({\bm u},f,\varphi)=({\bm u}_n,f_n,0)$ or $({\bm u},f,\varphi)=({\bm u}_n,0,\varphi_n)$. Then we have
\begin{align*}
\qquad  \frac{\int_{\re^N}|\triangle {\bm u}_n|^2|{\bm x}|^{2\gamma}dx}{\int_{\re^N}|\nabla{\bm u}_n|^2|{\bm x}|^{2\gamma-2}dx}
 &=\left\{
 \begin{aligned}
  &\frac{\int_{\re} h(\frac{{t}}{n})\hspace{0.1em}Q_0(-\pd_{t}^2)h(\frac{t}{n})dt}{\int_{\re} h(\frac{{t}}{n})\hspace{0.1em}P_0(-\pd_{t}^2)h(\frac{t}{n})dt} &\text{ if }\ \nu_0=0,\\
  &\frac{\int_{\re} h(\frac{t}{n})\hspace{0.1em}Q_1(-\pd_{t}^2,\alpha_{\nu_0})h(\frac{t}{n})\hspace{0.1em}dt}{\int_{\re} h(\frac{t}{n})\hspace{0.1em}P_1(-\pd_{t}^2,\alpha_{\nu_0})h(\frac{t}{n})\hspace{0.1em}dt} &\text{ otherwise}.
 \end{aligned}
 \right.
 \\&= 
 \left\{ 
 \begin{aligned}
  &\frac{\int_{\re} h({t})\hspace{0.1em}Q_0\(-n^{-2}\pd_{t}^2\)h({t})dt}{\int_{\re} h({t})\hspace{0.1em}P_0\(-n^{-2}\pd_{t}^2\)h({t})dt} &\text{ if }\ \nu_0=0  ,
  \\
  &\frac{\int_{\re} h({t})\hspace{0.1em}Q_1(-n^{-2}\pd_{t}^2,\alpha_{\nu_0})h\({t}\)\hspace{0.1em}dt}{\int_{\re} h({t})\hspace{0.1em}P_1(-n^{-2}\pd_{t}^2,\alpha_{\nu_0})h\({t}\)\hspace{0.1em}dt} &\text{ otherwise}.
 \end{aligned}\right.
\end{align*} 
Notice on the right-hand side that the denominator always exceeds a fixed positive number, since  $P_0(0)\ge N-1>0$ and $P_1(0,\alpha_{\nu_0})>0$ as mentioned above. 
Therefore, passing to $n\to\infty$, we get
\begin{align*}
\frac{\int_{\re^N}|\triangle {\bm u}_n|^2|{\bm x}|^{2\gamma}dx}{\int_{\re^N}|\nabla {\bm u}_n|^2|{\bm x}|^{2\gamma-2}dx}
&= \begin{cases}
	\dfrac{O(1/n^2)+Q_0(0)}{O(1/n^2)+P_0(0)}&\text{if }\ \nu_0=0
	\vspace{0.75em} \\
	\dfrac{O(1/n^2)+Q_1(0,\alpha_{\nu_0})}{O(1/n^2)+P_1(0,\alpha_{\nu_0})} &\text{otherwise}
	\end{cases} \\
	&\longrightarrow C_{N,\gamma,\nu_0}=C_{N,\gamma},
\end{align*}
which gives the desired sharpness of $C_{N,\gamma}$. 

Now, the proof of Theorem \ref{theorem:RH} has been completed.
\qed

\section{Proof of Theorem \ref{theorem:RH_c}}
\label{sec:RH_c}

Let $\nu_1$ denote the positive integer such that 
\[
 C_{N,\gamma,\nu_1}=\min_{\nu\in\mathbb{N}}C_{N,\gamma,\nu}.
\]
In order to estimate the difference between both sides of the inequality \eqref{RH_cf}, recall from the same calculation in the first line of \eqref{quotient:RH} the expression of the integrals:
\[
\begin{split}
 \int_{\mathbb{R}^N}|\triangle {\bm u}|^2|{\bm x}|^{2\gamma}dx&=\iint_{\re\times\mathbb{S}^{N-1}}\(Q_0(\tau^2)|\widehat{f}|^2+\sum_{\nu \in \N} Q_1(\tau^2,\alpha_\nu)|\widehat{\varphi_\nu}|^2  \)d\tau\hspace{0.1em}\mathrm{d}\sigma,
\\
 \int_{\mathbb{R}^N}\frac{|\nabla {\bm u}|^2}{|{\bm x}|^2}|{\bm x}|^{2\gamma}dx&= \iint_{\re\times\mathbb{S}^{N-1}}\(P_0(\tau^2)|\widehat{f}|^2+\sum_{\nu \in \N} P_1(\tau^2,\alpha_\nu)|\widehat{\varphi_\nu}|^2 \)d\tau\hspace{0.1em}\mathrm{d}\sigma.
\end{split}
\]
Then we have the following estimate:
\begin{align*}
\int_{\mathbb{R}^N}&|\triangle {\bm u}|^2|{\bm x}|^{2\gamma}dx-C_{N,\gamma}\int_{\mathbb{R}^N}\frac{|\nabla {\bm u}|^2}{|{\bm x}|^2}|{\bm x}|^{2\gamma}dx
\\
 &= \int_{\mathbb{R}^N}|\triangle {\bm u}|^2|{\bm x}|^{2\gamma}dx-\min\left\{\frac{Q_1(0,\alpha_{\nu_1})}{P_1(0,\alpha_{\nu_1})},\frac{Q_0(0)}{P_1(0)}\right\}\int_{\mathbb{R}^N}\frac{|\nabla{\bm u}|^2}{|{\bm x}|^2}|{\bm x}|^{2\gamma}dx
 \\&\ge \int_{\mathbb{R}^N}|\triangle {\bm u}|^2|{\bm x}|^{2\gamma}dx-\iint_{\mathbb{R}\times\mathbb{S}^{N-1}}\frac{Q_0(0)}{P_0(0)}P_0(\tau^2)|\widehat{f}|^2 
 \\&\quad -\iint_{\mathbb{R}\times\mathbb{S}^{N-1}}\frac{Q_1(0,\alpha_{\nu_1})}{P_1(0,\alpha_{\nu_1})}\sum_{\nu\in\mathbb{N}}P_1(\tau^2,\alpha_\nu)|\widehat{\varphi_\nu}|^2
 d\tau\mathrm{d}\sigma
 \\&= \iint_{\mathbb{R}\times\mathbb{S}^{N-1}}\left(Q_0(\tau^2)-\frac{Q_0(0)}{P_0(0)}P_0(\tau^2)\right)|\widehat{f}|^2d\tau\mathrm{d}\sigma 
 \\&\quad+
 \sum_{\nu\in\mathbb{N}}\iint_{\mathbb{R}\times\mathbb{S}^{N-1}}\left(Q_1(\tau^2,\alpha_\nu)-\frac{Q_1(0,\alpha_{\nu_1})}{P_1(0,\alpha_{\nu_1})}P_1(\tau^2,\alpha_{\nu})\right)|\widehat{\varphi_\nu}|^2d\tau\mathrm{d}\sigma
 \\&\ge \mathop{\iint}_{\re\times\mathbb{S}^{N-1}}P_0(\tau^2)\tau^2|\widehat{f}|^2d\tau\mathrm{d}\sigma
 +{c_0}\sum_{\nu\in\N} \mathop{\iint}_{\re\times\mathbb{S}^{N-1}}P_1(\tau^2,\alpha_{\nu})\tau^2|\widehat{\varphi_\nu}|^2d\tau\mathrm{d}\sigma 
 \\&\ge\min\left\{1,{c_0}\right\}\mathop{\iint}_{\re\times\mathbb{S}^{N-1}}\left(P_0(\tau^2)\tau^2|\widehat{f}|^2+\sum_{\nu\in\mathbb{N}}P_1(\tau^2,\alpha_\nu)\tau^2|\widehat{\varphi_\nu}|^2\right)d\tau\mathrm{d}\sigma 
\\&=\min\left\{1,{c_0}\right\}\mathop{\iint}_{\mathbb{R}\times\mathbb{S}^{N-1}}\Big((\partial f)P_0(-\partial^2)\partial f+(\partial\varphi) P_1(-\partial^2,-\triangle_\sigma)\partial\varphi\Big) d{t}\hspace{0.1em}\mathrm{d}\sigma
 \\&=\min\left\{1,{c_0}\right\}\int_{\mathbb{R}^N}\frac{\left|\nabla\(|{\bm x}|^{\lambda}\partial \(|{\bm x}|^{-\lambda}{\bm u}\)\)\right|^2}{|{\bm x}|^2}|{\bm x}|^{2\gamma}dx,
\end{align*}
where the last equation follows by applying the replacement \eqref{replace} to the integral equation in \eqref{poly:P}, and 
where we the fourth inequality follows by using the inequalities \eqref{est:Q0/P0} and
\[\begin{aligned}
\frac{1}{\tau^2}\left(\frac{Q_1(\tau^2,\alpha_\nu)}{P_1(\tau^2,\alpha_\nu)}-\frac{Q_1(0,\alpha_{\nu_1})}{P_1(0,\alpha_{\nu_1})}\right)
   &\ge  \frac{1}{\tau^2}\left(\frac{Q_1(\tau^2,\alpha_\nu)}{P_1(\tau^2,\alpha_\nu)}-\frac{Q_1(0,\alpha_\nu)}{P_1(0,\alpha_\nu)}\right)
\\&\ge {c_0}
\qquad\forall (\tau,\nu)\in\(\mathbb{R}\setminus\{0\}\)\times\mathbb{N},
  \end{aligned} 
\] 
as verified by using the same constant 
$c_0$ given in Lemma \ref{lemma:Q/P}. 
Finally, by restoring the notations \eqref{eta} and \eqref{Emden}, we obtain
\[
\begin{split}
 \int_{\mathbb{R}^N}|\triangle {\bm u}|^2|{\bm x}|^{2\gamma}dx&-C_{N,\gamma}\int_{\mathbb{R}^N}\frac{|\nabla {\bm u}|^2}{|{\bm x}|^2}|{\bm x}|^{2\gamma}dx\\&\ge {c}\int_{\mathbb{R}^N}\left|\nabla\(|{\bm x}|^{2-\frac{N}{2}-\gamma}({\bm x}\cdot\nabla) \big(|{\bm x}|^{\gamma+\frac{N}{2}-2}{\bm u}\big)\)\right|^2|{\bm x}|^{2\gamma-2}dx
\end{split}
\]
 for some absolute constant $c > 0$. 
 The proof of Theorem \ref{theorem:RH_c} is now complete, although the optimal value of ${c}$ is not known.
\qed

\section{An observation of the best constant $C_{N,\gamma}$ in Theorem \ref{theorem:RH}}
\label{sec:observation}

Concerning the constants in the inequalities \eqref{RH} and \eqref{RH_cf}, it holds that
\begin{equation}
\label{CA}
C_{N,\gamma}\ge A_{N,\gamma}
\end{equation} 
as a matter of course.
Here we wish to evaluate whether the strict inequality $C_{N,\gamma}>A_{N,\gamma}$ holds or not.
However, since the expression for $C_{N,\gamma}$ is complicated and its full picture seems difficult to reveal, we describe it for only some specific values of $\gamma$.  

\subsection{Preliminary: a review of $A_{N,\gamma,\nu}$}
In view of the original best constant \eqref{ANg} in Rellich-Hardy inequality, let us observe the increase or decrease of the function $\nu\mapsto A_{N,\gamma,\nu}$. 
In terms of the notation \eqref{eta},
  the expression of $A_{N,\gamma,\nu}$ in  \eqref{ANg} can be rewritten as
\begin{equation}
\label{ANg-}   A_{N,\gamma,\nu}=\left\{
		     \begin{array}{cl}
		      (\lambda+N-2)^2&\text{for }\nu=0,
		       \vspace{0.5em}
		       \\
		      \dfrac{(\alpha_\nu-\alpha_\lambda)^2}{\alpha_\nu+\lambda^2}&\text{for }\nu\in\mathbb{N}.
		     \end{array}
		    \right.
  \end{equation}
with the aid of \eqref{as-at}. 
Then a direct calculation from this expression yields
\begin{align*}
\tl{A1-A0}  &  \frac{A_{N,\gamma,1}-A_{N,\gamma,0}}{N-1} =-\frac{3\lambda^2+4(N-2)\lambda+N^2-5 N+5}{\lambda^2+N-1}
\intertext{and} 
 \quad &  \frac{A_{N,\gamma,\nu+1}-A_{N,\gamma,\nu}}{2 \nu+N-1}
 \\  &=
 \frac{\alpha_{\nu} \alpha_{\nu+1}+\lambda^2 \(2 \nu (\nu+N-1)+\big(1+\frac{1}{N-1}\lambda^2\big)\big(A_{N,\gamma,1}-A_{N,\gamma,0}\big)\) }{\(\alpha_\nu+\lambda^2\) \(\alpha_{\nu+1}+\lambda^2\)}
\tl{A-A}
\end{align*}
for all $\nu\in\mathbb{N}$. 
Notice 
that the numerator of the right-hand side is monotone increasing in $\nu\ge0$, as well as that the denominator is always positive. 
Therefore, for every $k\in\mathbb{N}\cup\{0\}$ the two inequalities
\begin{equation}
\label{A-Ak}
  A_{N,\gamma,k}\le A_{N,\gamma,k+1} \quad \text{and} \quad
  A_{N,\gamma,\nu}\le A_{N,\gamma,\nu+1} \quad (\forall\nu\ge k) 
\end{equation}
are equivalent. 
  \subsection{The case $\gamma=2-\frac{N}{2}$ (\tcr{or equivalently} $\lambda=0$)}
  Let us deal with this ``singular'' case, in the sense of \S \ref{subsec:P1_zero}. Following from the definitions of $A_{N,\gamma}$ and $C_{N,\gamma}$, 
we have
\begin{align*}
	 A_{N,2-\frac{N}{2}}&=\min\left\{A_{N,2-\frac{N}{2},0},\ \min_{\nu\in\mathbb{N}}A_{N,2-\frac{N}{2},\nu}\right\} \notag 
	\\&=\min\left\{\(2-N\)^2,\ \min_{\nu\in\mathbb{N}}\alpha_\nu\right\}=\min\Big\{(N-2)^2,\ \underbrace{N-1\strut}_{\quad =\,\alpha_1} \Big\} \notag
	 \\&=\begin{cases}
		  0&(N=2) \vspace{0.25em}\\ 
		N-1&(N\ge3) \\
	\end{cases},
	\\ \notag
	\tl{C_2-n/2} C_{N,2-\frac{N}{2}}&=\min\left\{C_{N,2-\frac{N}{2},0},\ C_{N,2-\frac{N}{2},1},\ \min_{\nu\in\mathbb{N}\setminus\{1\}}C_{N,2-\frac{N}{2},\nu}\right\} \notag
	 \\&=\min\left\{N-1,\ \tfrac{N^3}{3N-2},\ C_{N,2-\frac{N}{2},2}\right\} \notag
	 \\&=\min\left\{N-1,\ \tfrac{N^3}{3N-2},\ \tfrac{(N+1)(2N+1)}{2N-1}\right\} \notag
	 \\&=N-1. \notag
\end{align*}
Here the third equality from the last in \eqref{C_2-n/2} follows with the aid of computing
\[
C_{N,2-\frac{N}{2},\nu}=(\alpha_\nu+1)\(1-\tfrac{N-2}{\alpha_\nu-1}\)
\] 
which is monotone increasing in $\nu\ge2$. Summarizing the results above, we have obtained 
\[
\begin{split}
&  C_{2,1}=C_{2,1,0}=1>0=A_{2,1,0}=A_{2,1},
\\
 & C_{N,2-\frac{N}{2}}=C_{N,2-\frac{N}{2},0}=N-1=A_{N,2-\frac{N}{2},1}=A_{N,2-\frac{N}{2}}\quad(N\ge3).
\end{split}
\]
In particular, we see that the best constant in the two-dimensional Rellich-Hardy inequality (with $\gamma = 1$) can be really improved by the curl-free condition on the test vector fields.
  \subsection{The case $\gamma\ne 2-\frac{N}{2}$ (\tcr{or equivalently} $\lambda\ne0$ )}
Recall from \S\ref{subsec:lower} and \S\ref{subsec:P1_zero} that
\[
\begin{cases}
  C_{N,\gamma,\nu}=\dfrac{
 \(\alpha_\nu+(\lambda-1)^2\)(\alpha_{\lambda+1}-\alpha_\nu)^2}
 {P_1(0,\alpha_\nu)},
\vspace{0.5em}
\\ 
P_1(0,\alpha_\nu)=\(\alpha_{\lambda+1}-\alpha_\nu\)^2+\(2\lambda+N-2\)\big((2\lambda+1)\alpha_\nu-(N-1)(\lambda+1)^2\big)>0
\end{cases}
\]
for $\nu\in\mathbb{N}$. By using this expression together with \eqref{ANg-}, a direct computation yields
\begin{equation*}
 \begin{cases}
  \displaystyle
 \ C_{N,\gamma,\nu}-A_{N,\gamma,\nu-1}
  =-\frac{2 (\nu-\lambda-1)^2 W(\alpha_{\nu}) (\nu+N-2)}{(\alpha_{\nu-1}+\lambda^2) P_1(0,\alpha_{\nu})},
  \vspace{0.5em}
  \\
  \displaystyle
 \ C_{N,\gamma,\nu}-A_{N,\gamma,\nu+1}
  =\frac{2\nu W(\alpha_{\nu}) (1-\nu-N-\lambda)^2}{(\alpha_{\nu+1}+\lambda^2)P_1(0,\alpha_{\nu})} 
 \end{cases}
\end{equation*}
for all $\nu\in\mathbb{N}$, 
where
\[  W(\alpha_\nu)=\lambda^2 (2 \lambda+N-4) (2 \lambda+N)+(N-2)^2-\(\alpha_\nu+\lambda^2-1\)^2.
\]  
Then we have
\[\(C_{N,\gamma,\nu}-A_{N,\gamma,\nu-1}\)
 \(C_{N,\gamma,\nu}-A_{N,\gamma,\nu+1}\)
 \le 0\]
or equivalently
\begin{equation}
\label{ACA}  \min\{A_{N,\gamma,\nu-1},A_{N,\gamma,\nu+1}\}\le C_{N,\gamma,\nu}\le \max\{A_{N,\gamma,\nu-1},A_{N,\gamma,\nu+1}\}
\end{equation}
for all $\nu\in\mathbb{N}$. 
Based on this fact, let us consider the following two simplest cases:

\subsubsection*{The case $A_{N,\gamma}=A_{N,\gamma,1}$} 

It holds from \eqref{C0A1} that  \[A_{N, \gamma} = A_{N, \gamma,1} = C_{N, \gamma,0} \ge  C_{N, \gamma},\] and hence that $C_{N,\gamma}=A_{N,\gamma}$ from \eqref{CA}. This fact says that the curl-free restriction causes no effect on the improvement of the best constant. 
Now, we seek for when the equation $A_{N, \gamma} = A_{N, \gamma,1}$ will happen.
In view of the inequalities  \eqref{A-Ak} with $k=1$, this equation is equivalent to that both the inequalities
\[
A_{N,\gamma,1} \le A_{N,\gamma,0}\quad\text{ and }\quad A_{N,\gamma,1}\le A_{N,\gamma,2}
\]
hold true;  
in other words, $A_{N,\gamma}=A_{N,\gamma,1}$ holds if and only if both the numerators of the right-hand sides of \eqref{A1-A0} and \eqref{A-A}$_{\nu=1}$ evaluate to
\[
 \begin{split}
 &3\lambda^2+4(N-2)\lambda+N^2-5 N+5\ge0,\qquad\text{and}
  \\&\alpha_2 \alpha_1+\lambda^2 \(2N+\big(1+\tfrac{1}{N-1}\lambda^2\big)\(A_{N,\gamma,1}-A_{N,\gamma,0}\)\)
\\&\quad=-3 \lambda^4-4(N-2) \lambda^3- (N^2-7 N+5)\lambda^2+2N(N-1)
\\&\quad \ge0.
 \end{split}
\]
For example, if $\gamma=0$ \big(or equivalently $\lambda=2-\frac{N}{2}$\big), then the two inequalities become
\[
\begin{split}
 & 
N^2-4N-4
\le 0,\quad \text{ and}
\\& 
N^4-4N^3+12N^2+64N-64
\ge0\quad(\text{which is always true}). 
\end{split}
\]
This is the case when $N \le 4$, and hence we have
\[\begin{cases}
  C_{2,0}=C_{2,0,0}=A_{2,0,1}=A_{2,0}=0, 
\\ C_{3,0}=C_{3,0,0}=A_{3,0,1}=A_{3,0} = \frac{25}{36} ,  
 \\  C_{4,0}=C_{4,0,0}=A_{4,0,1}=A_{4,0} = 3,
\end{cases}
\]
which says that no curl-free improvement occurs when $N \in \{2, 3, 4 \}$ and $\gamma = 0$.

We may understand the above result in the following way:
we can choose a 
sequence $\{{\bm u}_n\}_{n\in\mathbb{N}}$ of vector fields 
by the formula
\[
 {\bm w}_n({\bm x})={\bm x}f_n(|{\bm x}|)\qquad \(n=1,2,\cdots\)
\]
in order that
\[
\begin{split}
 & \frac{\int_{\mathbb{R}^N}|\triangle {\bm u_n}|^2|{\bm x}|^{2\gamma}dx}{\int_{\mathbb{R}^N}|\nabla {\bm w}_n|^2|{\bm x}|^{2\gamma-2}dx}
\to A_{N,\gamma,1}
\qquad(n\to\infty),
\end{split}
\]
where $\{f_n({r})\}_{n\in\mathbb{N}}$ are a smooth functions vanishing near ${r}=0$. We may see this fact by noticing that 
 each coordinate function $x_k$ $(k=1,\cdots, N)$ satisfies the eigenequation $-\triangle_\sigma x_k =\alpha_1 x_k$. 
On the other hand, it is easy to check that $\{{\bm w}_n\}_{n\in\mathbb{N}}$ are always curl-free, whence $C_{N,\gamma}\le A_{N,\gamma}$. Therefore, we get  $C_{N,\gamma}=A_{N,\gamma}$ from $A_{N,\gamma}=A_{N,\gamma,1}$.

\subsubsection*{The case $A_{N,\gamma}=A_{N,\gamma,0}$} 

Since $A_{N,\gamma,1}\ge A_{N,\gamma,0} = A_{N, \gamma}$, we see from \eqref{A-Ak} and \eqref{ACA} that 
\[
 A_{N,\gamma,\nu-1}\le C_{N,\gamma,\nu}\le A_{N,\gamma,\nu+1}
\]
holds for all $\nu\in\mathbb{N}$. 
This fact implies
\[
	\min_{k\in\mathbb{N}\cup\{0\}}C_{N,\gamma,2k}=C_{N,\gamma,0}\quad\text{ and }\quad\min_{k\in\mathbb{N}}C_{N,\gamma,2k-1}=C_{N,\gamma,1},
\]
whence
\[
 C_{N,\gamma}=\min\left\{C_{N,\gamma,0},\ C_{N,\gamma,1}\right\}.
\]
On the other hand, a direct computation yields
\[
\begin{split}
 C_{N,\gamma,1}-C_{N,\gamma,0}
&=
 \frac{(\lambda-1)^2\(A_{N,\gamma,1}-A_{N,\gamma,0}\) +\frac{(N^2-1)}{5}\frac{(5 \lambda+2 N-4)^2+N^2+N-1}{ \lambda^2+N-1}}{(\lambda+1)^2+3(N-1)}
\\&>0
\end{split}
\]
from $A_{N,\gamma,1}\ge A_{N,\gamma,0}$. Therefore, it holds from \eqref{C0A1} that
\[
 C_{N,\gamma}=C_{N,\gamma,0}=A_{N,\gamma,1},
\]
or equivalently that
\[C_{N,\gamma}-A_{N,\gamma}=A_{N,\gamma,1}-A_{N,\gamma,0}.\]
This fact together with \eqref{A1-A0} shows that the inequality $C_{N,\gamma} > A_{N,\gamma}$ (namely the effect of the curl-free-improvement) holds  
as far as the right-hand side of \eqref{A1-A0} is strictly negative,   
or equivalently 
\[ 
	\left|\gamma-\tfrac{N+4}{6}\right|< \tfrac{1}{3}\sqrt{N^2-N+1}.
\]
For example, this is the case if $\gamma=0$ and $N\ge5$, whence we have  \[C_{N,0}=C_{N,0,0}=A_{N,0,1} = \tfrac{(\frac{N^2}{4}-1)^2}{\frac{N^2}{4}-N+3} > \tfrac{N^2}{4}=A_{N,0,0} = A_{N,0}\] 
for all $N \ge 5$. 

We may understand the above phenomenon in the following way: 
we can choose a sequence $\{{\bm w}_n\}_{n\in\mathbb{N}}$ of vector fields 
by the formula 
\[
{\bm w}_n({\bm x})=\(f_{n,1}(|{\bm x}|),\cdots,f_{n,N}(|{\bm x}|)\)\qquad(n=1,2,\cdots)
\]
in order that
\begin{equation}
 \label{RHtoA0} \frac{\int_{\mathbb{R}^N}|\triangle {\bm w_n}|^2|{\bm x}|^{2\gamma}dx}{\int_{\mathbb{R}^N}|\nabla {\bm w}_n|^2|{\bm x}|^{2\gamma-2}dx}\to A_{N,\gamma,0}
  \qquad(n\to\infty).
\end{equation}
In order for ${\bm w}_n$ to be curl-free, it must hold that
\[
 \frac{\partial f_{n,j}}{\partial x_k}-\frac{\partial f_{n,k}}{\partial x_j}
 =\frac{\partial {r}}{\partial x_k}f_{n,j}'-\frac{\partial{r}}{\partial x_j}f_{n,k}'=\frac{x_k}{{r}}f_{n,j}'-\frac{x_j}{{r}}f_{n,k}'=0, \quad \forall j, k \in \{1,\cdots, N\},
\]
where the notation $(\cdots)'$ denotes the derivative of a one-dimensional function. 
This fact implies that $f_{n,k}'=(x_k/x_j)f_{n,j}'$ for all $j\ne k$, and hence that  \[f_{n,j}={\rm const} \qquad  \forall j\in\{1,\cdots,N\}
\] from the radial symmetry of the function ${\bm x}\mapsto f_{n,j}(|{\bm x}|)$. Consequently,  we have $|\triangle {\bm w}_n|=|\nabla{\bm w}_n|\equiv 0$. 
This phenomenon indicates that there may be no curl-free sequence $\{{\bm w}_n\}_{n\in\mathbb{N}}$ satisfying \eqref{RHtoA0}, which we can naturally interpret as a result of the inequality 
$C_{N,\gamma}>A_{N,\gamma}$.

Incidentally, let us further consider the case
\[\begin{cases} 
1-\frac{N}{2}+\sqrt{N+1}\le\gamma\le\frac{N+4}{6}+\frac{1}{3}\sqrt{N^2-N+1}&\text{for }\ 3\le N\le 11\vspace{0.5em}
\\\left|\gamma-\frac{N+4}{6}\right|\le \frac{1}{3}\sqrt{N^2-N+1}&\text{for }\ N\ge 12
\end{cases}
\]
which ensures both $C_{N,\gamma}=C_{N,\gamma,0}$ and 
\[H_{N,\gamma-1}
=\(\gamma + \tfrac{N}{2}-2\)^2 + N-1\]from \eqref{HNg}. 
(For $\gamma=0$, this is the case for all $N\ge8$.) 
Then a successive application of Rellich-Hardy inequality and  Hardy-Leray inequality  reproduces  Rellich-Leray inequality:  for all curl-free fields ${\bm u}\in \mathcal{D}_{\gamma-1}(\mathbb{R}^N)$, it holds  that
\begin{alignat*}{3}  
	\int_{\mathbb{R}^N}|\triangle {\bm u}|^2|{\bm x}|^{2\gamma}dx& \ge C_{N,\gamma}\int_{\mathbb{R}^N}|\nabla {\bm u}|^2|{\bm x}|^{2\gamma-2}dx \qquad\qquad  (\text{from \eqref{RH_cf}})
	\\&\ge C_{N,\gamma}H_{N,\gamma-1} \int_{\mathbb{R}^N}\frac{|{\bm u}|^2}{|{\bm x}|^2}|{\bm x}|^{2\gamma-2}dx \quad 
\(\begin{array}{ll}\small
\text{from \eqref{HL}}\\\small
\text{with $\gamma$ replaced by $\gamma-1$} \end{array}\)
	\\&=C_{N,\gamma,0}\(\(\gamma+\tfrac{N}{2}-2\)^2+N-1\)\int_{\mathbb{R}^N}\frac{|{\bm u}|^2}{|{\bm x}|^2}|{\bm x}|^{2\gamma-2}dx 
 \\&=\((\gamma-1)^2-\tfrac{N^2}{4}\)^2\int_{\mathbb{R}^N}\frac{|{\bm u}|^2}{|{\bm x}|^2}|{\bm x}|^{2\gamma-2}dx
	\\&\ge B_{N,\gamma}\int_{\mathbb{R}^N}\frac{|{\bm u}|^2}{|{\bm x}|^4}|{\bm x}|^{2\gamma}dx 
\end{alignat*}
with $B_{N,\gamma}$ given by \eqref{RL}. 
Therefore, even under the curl-free constraint, Rellich-Hardy inequality bridges between \tcr{Hardy-Leray} and Rellich-Leray inequalities, and so serves as a stronger version of the Rellich-Leray inequality. 


\section{Completion of the proof of Lemma \ref{lemma:Q/P}}
\label{sec:Q/P}

In this section, we prove Lemma \ref{lemma:Q/P} for the case $\gamma>1$ (\tcr{or equivalently} $\lambda<1-\frac{N}{2}$) with the aid of Maxima, by a similar technique to the former case.
More precisely, our goal is to show the two inequalities:
 \begin{align*}
  &\frac{1}{\tau}\(\frac{Q_1(\tau,{a})}{P_1(\tau,{a})}-\frac{Q_1(0,{a})}{P_1(0,{a})}\)\ge \frac{1}{2}\qquad\text{when } N\ge3,
\\&\frac{1}{\tau}\(\frac{Q_1(\tau,{a})}{P_1(\tau,{a})}-\frac{Q_1(0,{a})}{P_1(0,{a})}\)\ge \frac{1}{3}\qquad\text{when }N=2 
 \end{align*}
for all ${a}\in\{\alpha_\nu\}_{\nu\in\mathbb{N}}$,
which clearly satisfies the conclusion of the lemma.
\subsection{The case $N\ge3$}
By adding $1/2$ to both sides of \eqref{q/p:1}, we have
\begin{align*}
 \frac{1}{\tau}&\(\frac{Q_1(\tau,{a})}{P_1(\tau,{a})}-\frac{Q_1(0,{a})}{P_1(0,{a})}\)-\frac{1}{2}
\\&= (2 \lambda+N-2) \frac{ \ G_0({a})+G_1({a}) \tau}{P_1(0,{a})P_1(\tau,{a})}+\frac{1}{2}
 \\&=  \frac{ (2 \lambda+N-2)\( G_0({a})+G_1({a}) \tau\)}{P_1(0,{a})\Big(P_1(0,{a})+\tau^2+\big(2\({a}+\lambda^2+\lambda\)+N\big)\tau\Big)}+\frac{1}{2}
 \\&=\frac{\tau^2 P_1(0,{a})+\tau E_1({a})+E_0({a})}{2P_1(\tau,{a})P_1(0,{a})},
\tl{q/p:2}  
\end{align*}
where
\begin{align*}
\tl{E1}  
 E_1({a})&:=2(2\lambda+N-2)G_1({a})+\(2({a}+\lambda^2+\lambda)+N\)P_1(0,{a}),
\\
\tl{E0}E_0({a})&:=2\(2\lambda+N-2\)G_0({a})+\(P_1(0,{a})\)^2
 \end{align*}
Then it suffices to check that
 \begin{align*}
  E_1({a})&\ge0\quad\text{ and }\quad E_0({a})\ge0\qquad\forall{a}\in\{\alpha_\nu\}_{\nu\in\mathbb{N}}
\intertext{\qquad or 
that}
\tl{E1>0}E_1(\alpha_1+{s})&\ge0\quad\text{ and }\quad E_0(\alpha_1+{s})\ge0\qquad\forall{s}\ge0.
 \end{align*}
In the following, let us check the two inequalities in \eqref{E1>0}.

\subsubsection*{Proof of $E_1(\alpha_1+{s})\ge0$ } 
A direct computation from \eqref{poly:P} yields
\begin{align*}
 P_1(0,\alpha_1+{s})&=(\alpha_1+{s})^2+\(2(\lambda^2-\lambda)-N\)(\alpha_1+{s})
 \\&\quad +(\lambda+1)^2\(\lambda^2+N-1\)
 \\&=s^2+s\(2\lambda^2-2\lambda+N-2\)+\lambda^2\((\lambda+1)^2+3N-3\).
\tl{Taylor:P}
\end{align*}
Substitute \eqref{Taylor:G1} and \eqref{Taylor:P} into \eqref{E1} with ${a}=\alpha_1+{s}$; then a straightforward computation yields
\[
 \begin{split}
  E_1(\alpha_1+{s})&=
2(2\lambda+N-2)G_1(\alpha_1+{s})+\(2(\alpha_1+{s}+\lambda^2+\lambda)+N\)P_1(0,\alpha_1+{s})
\\&=2(2\lambda+N-2)\( \begin{array}{l}
		       s^2(2\lambda+N)
		      \vspace{0.25em}		      \\ +\,s\((N-1)(2\lambda+N-2)+(\lambda-1)^2(2\lambda+N)\)\vspace{0.25em}\\ +\,2(N-1)\lambda^2(2\lambda+N-1)\end{array}
\)
\\&\quad+\Big(2{s}+2\(\lambda^2+\lambda\)+3N-2\Big)
\(\begin{array}{l}
  s^2+s\(2\lambda^2-2\lambda+N-2\)
\vspace{0.25em}
  \\ +\,\lambda^2\((\lambda+1)^2+3N-3\)\end{array}
\)
  \\&=2{s}^3+{s}^2E_{12}(\lambda)+{s}E_{11}(\lambda)+\lambda^2E_{10}(\lambda)
\qquad\forall{s}\ge0
 \end{split}
\]
as a Taylor series of $E_1({a})$ at $a=\alpha_1$, where we set
\[
 \begin{split}
  E_{12}(\lambda)&:=2(2\lambda+N-2)(2\lambda+N)+2\(\lambda^2+\lambda\)+3N-2
\\&\quad\; +2\(2\lambda^2-2\lambda+N-2\)
\\&\;= 14\lambda^2+2(4N-5)\lambda+(N+2)(2N-3)
  \\&\;=
  2\(2\lambda+N-\tfrac{5}{4}\)^2+6\lambda^2+6(N-2)+\tfrac{23}{8}
,\\
E_{11}(\lambda)
&:=2(2\lambda+N-2)\Big((N-1)(2\lambda+N-2)+(\lambda-1)^2(2\lambda+N)\Big)
  \\&\quad\; +2\lambda^2((\lambda+1)^2+3N-3)+\Big(2\(\lambda^2+\lambda\)+3N-2\Big)\(2\lambda^2-2\lambda+N-2\)
\\&\; =14\lambda^4+4(2N-5)\lambda^3+2N(N+1)\lambda^2
\\&\quad\; +4(N-1)(N-2)\lambda +\(2N^2-N+2\)(N-2)
  \\&\;=6\lambda^4+\tfrac{1}{2}\lambda^2\Big((4\lambda+2N-5)^2+16N-9\Big)
\\&\quad\; +(N-2)\Big((2\lambda+N-1)^2+N^2+N+1\Big)
,
\\ E_{10}(\lambda)&:=
4(N-1)(2\lambda+N-2)(2\lambda+N-1)
\\&\quad\; +\(2(\lambda^2+\lambda)+3N-2\)\((\lambda+1)^2+3N-3\)
  \\&\;=(\lambda+1)^4+(\lambda+1)^2\lambda^2+(N-1)\(\(5\lambda+\tfrac{8N-6}{5}\)^2+\tfrac{36N^2+21N+89}{25}\).
 \end{split}
\]
Notice that $E_{12}(\lambda),E_{11}(\lambda)$ and $E_{10}(\lambda)$ are all nonnegative; then $E_1(\alpha_1+{s})\ge 0$ holds from ${s}\ge0$, which proves the desired inequality. 
\subsubsection*{Proof of $E_0(\alpha_1+{s})\ge0$ } 
Substitute \eqref{Taylor:G0} and \eqref{Taylor:P} into \eqref{E0} with ${a}=\alpha_1+{s}$; then a straightforward computation yields
\begin{align*}
 E_0(\alpha_1+{s})&=2\(2\lambda+N-2\)G_0(\alpha_1+{s})+\(P_1(0,\alpha_1+{s})\)^2
\\&=2(2\lambda+N-2)
\(\begin{array}{l}
      s^3(2\lambda+N)+s^2\mathcal{G}_2(\lambda) +s\mathcal{G}_1(\lambda)
  \vspace{0.25em}\\ +\,2(N-1)\lambda^4(2\lambda+N-1)\end{array}
\)
 \\&\quad+\Big(s^2+s\(2\lambda^2-2\lambda+N-2\)+\lambda^2\((\lambda+1)^2+3N-3\)\Big)^2
\\  &={s}^4+{s}^3E_{03}(\lambda)
  +{s}^2E_{02}(\lambda)+2\hspace{0.1em}{s}E_{01}(\lambda)
 +\lambda^4E_{00}(\lambda)
\qquad\forall{s}\ge0\quad 
\tl{E0_s}
 \end{align*}
as a Taylor series of $E_0({a})$ at ${a}=\alpha_1$,  
where we set
 \begin{align*}
  E_{03}(\lambda)&:=2(2\lambda+N-2)(2\lambda+N)+2\(2\lambda^2-2\lambda+N-2\)
  \\&\;=\tfrac{1}{2}(4\lambda+2N-3)^2+4\lambda^2+4(N-3)+\tfrac{7}{2},
  \\
E_{02}(\lambda)
  &:=2(2\lambda+N-2)\mathcal{G}_2(\lambda)+\(2\lambda^2-2\lambda+N-2\)^2
  \\&\quad
  +2\lambda^2\((\lambda+1)^2+3N-3\),\tl{E02}
  \\
\tl{E01}  E_{01}(\lambda)
  &:=(2\lambda+N-2)\mathcal{G}_1(\lambda)
  +\lambda^2\((\lambda+1)^2+3N-3\)\(2\lambda^2-2\lambda+N-2\),
  \\
  E_{00}(\lambda)&:=4(N-1)(2\lambda+N-2)(2\lambda+N-1)+\((\lambda+1)^2+3N-3\)^2
  \\&\;=(\lambda+1)^4+(N-1)\Big((4\lambda+2N-3)^2+6(\lambda+1)^2+9N-10
  \Big).
 \end{align*}
 Hence, in order to show $E_0(\alpha_1+{s})\ge0$ $(\forall{s}\ge0)$,  it suffices to check the nonnegativity of $\{E_{0k}(\lambda)\}_{k=0,1,2,3}$. Since it is clear that $E_{03}(\lambda)\ge0$ and $E_{00}(\lambda)\ge0$, all that is left is to show the two inequalities
\begin{equation}
 E_{02}(\lambda)\ge0\quad\text{ and }\quad E_{01}(\lambda)\ge0.
\label{E02E01}
\end{equation}
To this end, substitute \eqref{G2} into \eqref{E02}, then we get
\[
\begin{split}
 E_{02}(\lambda)&=2(2\lambda+N-2)\Big( (2\lambda+N-2)\((\lambda+1)^2+\lambda^2+N+3\)+(N-1)^2+9\Big)
\\&\quad+\(2\lambda^2-2\lambda+N-2\)^2
 +2\lambda^2\((\lambda+1)^2+3N-3\)
\\&=22\lambda^4+4(4N-5)\lambda^3+2(N+4)(2N+1)\lambda^2+4N\(4N-3\)\lambda
\\&\quad +\(4N^2+N+2\)(N-2)
 \\&=2\lambda^4+5\lambda^2\(2\lambda+\tfrac{4N}{5}-1\)^2
+N\(5\lambda+\tfrac{8N-6}{5}\)^2
 \\&\quad +\lambda^2\(\tfrac{4}{5}N^2+N+3\)
 +\tfrac{1}{25}(N-3)\(36N^2+29N+51\)+\tfrac{53}{25}.
\end{split}
\]
Since $N\ge3$, this result implies $E_{02}(\lambda)\ge0$, whence we have proved the first inequality of \eqref{E02E01}. In the same way, the proof of the second could also be carried out by the method of completing the square (see Appendix); however, we work here in a different way.  
From the assumption $\lambda<1-\frac{N}{2}$, we see that  $\lambda$ can be parameterized as
\[
 \lambda=1-\frac{N}{2}-{s},\qquad{s}>0.
\]
Then we have
\[
\begin{split}
  \mathcal{G}_{1}(\lambda)&=\mathcal{G}_{1}\(1-\tfrac{N}{2}-{s}\)
\\&=(-6-2{s})\(1-\tfrac{N}{2}-{s}\)^4+2\(1-\tfrac{N}{2}-{s}\)^2\(4{s}-2-2N+N^2\)+2N^2(1-{s})
\end{split}
\]
from \eqref{G1}. Substituting this expression into \eqref{E01}, we compute
\[
 \begin{split}
  E_{01}(\lambda)&=E_{01}\(1-\tfrac{N}{2}-{s}\)
  \\&=-2{s}\(\begin{array}{l}
	     (-6-2{s})\(1-\tfrac{N}{2}-{s}\)^4+2N^2(1-{s})\vspace{0.25em}\\ +\,2\(1-\tfrac{N}{2}-{s}\)^2\(4{s}-2-2N+N^2\)\end{array}
\)
  \\&\quad +\(1-\tfrac{N}{2}-{s}\)^2
  \Big({s}^2+(N-4){s}+\tfrac{1}{4}(N+2)^2\Big)
  \Big(2{s}^2+2(N-1){s}+\tfrac{1}{2}\(N^2-4\)\Big)
\\&=6{s}^6+2(7N-9){s}^5+\tfrac{1}{2}\big((N-3)(27N+25)+55\big){s}^4
 \\&\quad  +\Big((N-3)\(7N^2+2N-14\)+34\Big){s}^3
  \\&\quad +\tfrac{1}{8}\Big((N-3)^2\(17N^2+46N+107\)+740(N-3)+601\Big){s}^2
  \\&\quad +\Big((N-4)\(\tfrac{3}{8}N^2(N^2+N+4)+11N+14\)+86\Big){s}+\tfrac{1}{32}\(N^2-4\)^3
 \end{split}
\]
as a Taylor series of  $E_{01}(\lambda)$ at $\lambda=1-\frac{N}{2}$. When $N\ge4$, 
the coefficients of the powers of ${s}$ are all positive,  and hence $E_{01}(\lambda)\ge0$ holds true.  Moreover, even when  $N=3$, we directly see that
\[
\begin{split}
E_{01}(\lambda)&=E_{01}\(1-\tfrac{3}{2}-{s}\)
 =6{s}^6+24{s}^5+\tfrac{55}{2}{s}^4+34{s}^3+\underbrace{\tfrac{601}{8}{s}^2-15{s}+\tfrac{125}{32}\strut}_{\quad >\,0}
\\&\ge0.
\end{split}
\]
Therefore, we have obtained the second inequality of \eqref{E02E01} for all $N\ge3$, as desired. 
\subsection{The case $N=2$}
Applying the same calculation in \eqref{q/p:1} to the case $N=2$, we have the identity
\begin{align*}
 \frac{1}{\tau}\(\frac{Q_1(\tau,{a})}{P_1(\tau,{a})}-\frac{Q_1(0,{a})}{P_1(0,{a})}\)-1=\,  \frac{ 2\lambda\( G_0({a})+G_1({a}) \tau\)}{P_1(0,{a})P_1(\tau,{a})},
\end{align*}
where
\[
  P_1(\tau,{a})=
  \tau^2+2\({a}+\lambda^2+\lambda+1\)\tau+\({a}+\lambda^2-\lambda-1\)^2+(4\lambda+3)\lambda^2
\]
is the same as in \eqref{poly:P} with $N=2$,  
and where $G_0({a})$ and $G_0({a})$ are the same as in \eqref{G0G1} with $N=2$. 
Adding $2/3$ to both sides of the above identity, we then get
\[
\begin{split}
  \frac{1}{\tau}&\(\frac{Q_1(\tau,{a})}{P_1(\tau,{a})}-\frac{Q_1(0,{a})}{P_1(0,{a})}\)-\frac{1}{3}
 \\&=\frac{2}{3}\(1+  \frac{ 3\lambda \(G_0({a})+G_1({a}) \tau\)}{\Big(\tau^2+2\big({a}+\lambda^2+\lambda+1\big)\tau\Big)P_1(0,{a})+\(P_1\(0,{a}\)\)^2}\)
 \\&=\frac{2}{3}\cdot \frac{\tau^2P_1(0,{a})+2\tau F_1({a})+F_0({a})}{P_1(0,{a})P_1(\tau,{a})},
\end{split}
\]
where
\begin{align}
 \label{F1} F_1({a})&:=\({a}+\lambda^2+\lambda+1\)P_1(0,{a})+\tfrac{3}{2}\lambda G_1({a})
 ,\\
 \label{F0} F_0({a})&:=\(P_1(0,{a})\)^2+3\lambda G_0({a}).
\end{align}
Thus, it is enough to show that
\[
 F_1({a})\ge0\quad\text{ and }\quad
F_0({a})\ge0\qquad\forall{a}\in\{\alpha_\nu=\nu(\nu+1)\}_{\nu\in\mathbb{N}}
\]
under the assumption $\lambda<1-\frac{N}{2}=0$. 

\subsubsection*{Proof of $F_1({a})\ge0$ for ${a}\in\{\alpha_\nu\}_{\nu\in\mathbb{N}}$}It suffices to check that $F_1(\alpha_1+{s})\ge0$, i.e., \[F_1(1+{s})\ge0\qquad   \forall{s}\ge0.\] In order to compute the left-hand side, apply $N=2$ to \eqref{Taylor:P} and \eqref{Taylor:G1}; then we get
\begin{align*}
 \tl{Taylor:P,n=2} P_1(0,1+{s})&={s}^2+2{s}\lambda(\lambda-1)+\lambda^2\((\lambda+1)^2+3\),
 \\
 \tfrac{1}{2} G_1(1+{s})&=(\lambda+1){s}^2+{s}\(\lambda+(\lambda-1)^2(\lambda+1)\)+\lambda^2(2\lambda+1).
\end{align*}
Substitute them into \eqref{F1} and \eqref{F0} with ${a}=1+{s}$, then we have
\[
 \begin{split}
  F_1(1+{s})&=\({s}+\lambda^2+\lambda+2\)P_1(0,1+{s})+\tfrac{3}{2}\lambda G_1(1+{s})
  \\&=\({s}+\lambda^2+\lambda+2\)\Big({s}^2+2{s}\lambda(\lambda-1)+\lambda^2\((\lambda+1)^2+3\)\Big)
  \\&\quad +3\lambda\Big((\lambda+1){s}^2+{s}\(\lambda+(\lambda-1)^2(\lambda+1)\)+\lambda^2(2\lambda+1)\Big)
  \\&={s}^3+2\(3\lambda^2+\lambda+1\){s}^2+\lambda(6\lambda-1)\(\lambda^2+1\){s}
  \\&\quad +\lambda^2\(\lambda^4+3\lambda^3+14\lambda^2+11\lambda+8\)
  \\&={s}^3+\(6\(\lambda+\tfrac{1}{6}\)^2+\tfrac{11}{6}\){s}^2-\lambda(1-6\lambda)\(\lambda^2+1\){s}
  \\&\quad +\lambda^2\((\lambda+1)^2
 \(\lambda+\tfrac{1}{2}\)^2+\tfrac{27}{4}(\lambda+1)^2+(2\lambda-1)^2\)
 \end{split}
\]
In view of the right-hand side, the coefficients of the powers of ${s}$ are all positive since $\lambda<0$, which implies that $F_1({a})\ge0$ for all ${a}\ge\alpha_1$, as desired.
\subsubsection*{Proof of $F_0({a})\ge0$ for ${a}\in\{\alpha_\nu\}_{\nu\in\mathbb{N}}$}
First of all, let us compute $F_0(1+{s})$. To this end, apply $N=2$ to \eqref{Taylor:G0}, \eqref{G2} and \eqref{G1};  then we get
\[
\begin{split}
 \tfrac{1}{2}G_0(1+{s})&=(\lambda+1){s}^3+\tfrac{1}{2}{s}^2\mathcal{G}_2(\lambda)+\tfrac{1}{2}{s}\mathcal{G}_1(\lambda)+\lambda^4(2\lambda+1)
 \\&=(\lambda+1){s}^3+{s}^2\(2\lambda^3+2\lambda^2+6\lambda+5\)
 \\&\quad +{s}(\lambda-1)\(\lambda^4-2\lambda^3+6\lambda^2-8\lambda-4\)+\lambda^4(2\lambda+1).
\end{split}
\]
Substitute this expression and \eqref{Taylor:P,n=2} into \eqref{F0} with ${a}=1+{s}$, then we have
 \begin{align*}
    F_0(1+{s})&=\(P_1(0,1+{s})\)^2+3\lambda G_0(1+{s})
  \\&=\Big({s}^2+2{s}\lambda(\lambda-1)+\lambda^2\((\lambda+1)^2+3\)\Big)^2
  \\&\quad +3\lambda\Big(2(\lambda+1){s}^3+{s}^2\mathcal{G}_2(\lambda)+{s}\mathcal{G}_1(\lambda)+2\lambda^4(2\lambda+1)\Big)
  \\&=\Big({s}^2+2{s}\lambda(\lambda-1)+\lambda^2\((\lambda+1)^2+3\)\Big)^2\\&\quad +6\lambda\(\begin{array}{l}
   (\lambda+1){s}^3+{s}^2\(2\lambda^3+2\lambda^2+6\lambda+5\)+\lambda^4(2\lambda+1)
 \vspace{0.25em}\\ +\,{s}(\lambda-1)\(\lambda^4-2\lambda^3-6\lambda^2-8\lambda-4\)  \end{array}\)
  \\&={s}^4+2\lambda(5\lambda+1){s}^3+2\lambda(9\lambda^3+4\lambda^2+24\lambda+15){s}^2
  \\&\quad +2\lambda(\lambda-1)\(5\lambda^4-2\lambda^3-10\lambda^2-24\lambda-12\){s}
  \\&\quad +\lambda^4\((\lambda+1)^4+9(\lambda+1)^2+9\lambda^2+6\)
\tl{Taylor:F0}
\end{align*}
as a Taylor series of $F_0({a})$ at ${a}=\alpha_1=1$. Then it is clear that $F_0(\alpha_1)=F_0(1)\ge0$, and hence our goal is reduced to showing $F_0({a})\ge0$ for ${a}\ge\alpha_2$. For this purpose,  it suffices to check that
\[
F_0(\alpha_2+{s})= F_0(4+{s})\ge0\qquad\forall{s}\ge0.
\]
To this end, notice from \eqref{Taylor:F0} that
\[\begin{split}
   \frac{F_0(1+{s})}{{s}}
   &\ge \frac{F_0(1+{s})-F_0(1)}{{s}}
   \\&={s}^3+2\lambda
\(\begin{array}{l}
(5\lambda+1){s}^2+\(9\lambda^3+4\lambda^2+24\lambda+15\){s}
\vspace{0.25em}
\\ +\,(\lambda-1)\(5\lambda^4-2\lambda^3-10\lambda^2-24\lambda-12\)
\end{array}\)
\end{split}
\] 
holds for all ${s}>0$. Replacing ${s}$ by $3+{s}$ on both sides, we get
\[
 \begin{split}
  \frac{F_0(4+{s})}{3+{s}}&\ge ({s}+3)^3
  +2\lambda\(\begin{array}{l}
(5\lambda+1)({s}+3)^2
\vspace{0.25em}
\\+\,\(9\lambda^3+4\lambda^2+24\lambda+15\)({s}+3)
\vspace{0.25em}
\\ +\,(\lambda-1)\(5\lambda^4-2\lambda^3-10\lambda^2-24\lambda-12\)
	     \end{array}\)
  \\&=s^3+\((\lambda+1)^2+9\lambda^2+8\)s^2
  \\&\quad+\Big(14\lambda^4+\lambda^2\(4(\lambda+1)^2+83\)+21(\lambda+1)^2+6\Big)s
  \\&\quad +3\lambda^6+\lambda^4\(7(\lambda-1)^2+27\)+\lambda^2(2\lambda-1)^2
\\&\quad +16(4\lambda+1)^2+(\lambda+2)^2+7
\\&\ge0,
 \end{split}
\]
and therefore $F_0(4+{s})\ge0$ for all ${s}\ge0$, as desired.

\section*{Appendix}
Here we give another proof of the second inequality $E_{01}(\lambda)\ge0$ of \eqref{E02E01}, by the method of completing the square. 

Substitute \eqref{G1} into \eqref{E01}, then we have
\[\begin{split}
   E_{01}(\lambda)&=(2\lambda+N-2)\(\begin{array}{l}  \lambda^4(2\lambda+N-8)+\,N^2(2\lambda+N)\vspace{0.25em}\\+\,2\lambda^2\(2-4\lambda-4N+N^2\)  \end{array}\)\\&\quad +\lambda^2\((\lambda+1)^2+3N-3\)\(2\lambda^2-2\lambda+N-2\)  \\&=2\lambda^6+\tfrac{1}{4}\Big((4\lambda+2N-9)^2+8N-25\Big)\lambda^4\\&\quad +\Big(N(4N-1)+2\Big)\lambda^2+N^2(N-1)(\lambda+2)^2+4N^2\\&\quad +(N-6)\Big((2\lambda+N-1)^2\lambda^2+N^3\Big)
 \end{split}\]
by a straightforward computation, whence we get
\[ E_{01}(\lambda)\ge0\quad\text{ for all }N\ge6. \]
Moreover, the same also applies to the case $N\in\{3,4,5\}$, as can be directly verified by the following calculation:
\[\begin{split}
\text{when }N=3,\quad   
E_{01}(\lambda)&=6\lambda^6-6\lambda^5+10\lambda^4-24\lambda^3+41\lambda^2+72\lambda+27 
\\&=\tfrac{1}{63}\(\lambda+\tfrac{2}{3}\)^2\(40+\(28\lambda-\tfrac{145}{3}\)^2\) \\&\quad +\(\lambda+\tfrac{2}{3}\)^4\(\tfrac{1}{18}+\tfrac{1}{6}(6\lambda-11)^2\) +\tfrac{26}{9}\(3\lambda+\tfrac{11}{9}\)^2+\tfrac{1405}{729};
\\\text{when }N=4,\quad   
E_{01}(\lambda)&=6\lambda^6-2\lambda^5-6\lambda^4-24\lambda^3+92\lambda^2+192\lambda+128 
   \\&=\(2\lambda^3-\tfrac{1}{2}\lambda^2-8\lambda-8\)^2+2\lambda^6 +\tfrac{103}{4}\lambda^4\\&\quad +4\lambda^2+16(\lambda+2)^2;\\\text{when }N=5,\quad   
E_{01}(\lambda)&=6\lambda^6+2\lambda^5-16\lambda^3+181\lambda^2+400\lambda+375  
   \\&=5\(\lambda^3-\tfrac{3}{2}\)^2+\lambda^2\(\lambda^2+\lambda-\tfrac{1}{2}\)^2  \\&\quad +(\tfrac{25}{2}\lambda+16)^2+\tfrac{49}{2}\lambda^2+\tfrac{431}{4}. 
\end{split}\]
Therefore, we have obtained $E_{01}(\lambda)\ge0$ for all $N\ge3$, as desired.
%
%

\vspace{1em}\noindent
{\bf Acknowledgments.}
\\
The second author (F.T.) was supported by JSPS Grant-in-Aid for Scientific Research (B), No.19H01800.
This work was partly supported by Osaka City University Advanced Mathematical Institute 
(MEXT Joint Usage/Research Center on Mathematics and Theoretical Physics JPMXP0619217849).


\providecommand{\bysame}{\leavevmode\hbox to3em{\hrulefill}\thinspace}
\providecommand{\MR}{\relax\ifhmode\unskip\space\fi MR }
\providecommand{\MRhref}[2]{%
  \href{http://www.ams.org/mathscinet-getitem?mr=#1}{#2}
}
\providecommand{\href}[2]{#2}

\end{document}